\documentclass[11pt]{article}
\usepackage[margin=4.46cm]{geometry}
\usepackage{graphicx}
\usepackage{epsfig}
\usepackage{amsmath,amsthm}
\usepackage{amssymb}
\usepackage{caption}

\usepackage{tikz}
\numberwithin{equation}{section}
\newcommand{\commentout}[1]{}

\def \Rset {{\mathbb R}}

\def \Zset {{\mathbb Z}}

\def \Nset {{\mathbb N}}

\newcommand{\be}{\begin{equation}}
\newcommand{\ee}{\end{equation}}
\newcommand{\ba}{\begin{eqnarray}}
\newcommand{\ea}{\end{eqnarray}}
\newcommand{\bi}{\begin{itemize}}
\newcommand{\ei}{\end{itemize}}
\newcommand{\br}{\begin{eqnarray}}
\newcommand{\er}{\end{eqnarray}}

\newtheorem{theo}{Theorem}[section]
\newtheorem{defin}{Definition}[section]

\newtheorem{lem}{Lemma}[section]
\newtheorem{cor}{Corollary}[section]
\newtheorem{rmk}{Remark}[section]

\usepackage[colorlinks=true, pdfstartview=FitV, linkcolor=black, citecolor=blue, urlcolor=blue]{hyperref}

\definecolor{darkgreen}{rgb}{0,0.5,0}
\definecolor{darkblue}{rgb}{0,0,0.7}
\definecolor{darkred}{rgb}{0.9,0.1,0.1}
\definecolor{lightblue}{rgb}{0,0.51,1}

\newcommand{\wj}[1]{{\color{darkred}{#1}}}

\begin{document}
\title{\Large Does Yakhot's 
growth law for\\ turbulent burning velocity hold?}
\author{Wenjia Jing\thanks{Yau Math Sciences Center,
Tsinghua University, Beijing 100084, China.}, Jack Xin\thanks{Department of Mathematics, UC Irvine, Irvine, CA 92697, USA. }, and Yifeng Yu$^{\dagger}$}

\date{}
\maketitle
\thispagestyle{empty}
\begin{abstract}
Using formal renormalization theory,  Yakhot 
derived in (\cite{Yak}, 1988)  an $O\left(\frac{A}{\sqrt{\log A}}\right)$ growth  law of the turbulent flame speed with respect to large flow intensity $A$ based on the inviscid G-equation.  Although this growth law is widely cited in combustion literature, there has been no rigorous mathematical discussion to date about its validity. As a first step towards unveiling the mystery,  
we prove that  there is no intermediate growth law between $O\left(\frac{A}{\log A}\right)$ and $O(A)$  for  two dimensional incompressible Lipschitz continuous periodic flows with bounded swirl sizes. In particular, we do not assume the non-degeneracy of critical points. Additionally, other examples of flows with lower regularity, Lagrangian chaos, and related phenomena are also discussed. 
\end{abstract}
\thispagestyle{empty}

\hspace{.1 in} {\bf Keywords:}
Lipschitz continuous flows, level set G-equation,

\hspace{.1 in} control formula/paths, travel times, front speed growth laws. 
\medskip

\hspace{.1 in} {\bf MSC2020}: 35B27, 35B40, 35F21.

\newpage 
\setcounter{page}{1}
\section{Introduction} A central problem in the study of turbulent combustion is ``{\it how fast can it burn?}" (\cite{R}) due to its close connection with the efficiency of combustion engines.   In particular, it is important  to understand how the increase of the flow intensity  could enhance the turbulent flame speed $s_{\rm T}$.   This  has been studied extensively in combustion literature via theoretical, direct numerical simulations (DNS) and experimental approaches. For theoretical study, a common approach, called passive scalar models, is to decouple fluid and chemical reaction in the combustion process by prescribing the fluid velocity.   A popular platform to do this by ``{\it pencil and paper}"  is the so called G-equation model, which we now provide a brief review below. 

The G-equation is based on the simplest motion law that prescribes the normal velocity ($v_n$) of the moving interface
 to be the sum of the local burning speed ($s_l$)
and the projection of the fluid velocity $V$ along the normal $n$:
$$
V_n=s_l+V(x)\cdot n.
$$

\begin{figure}[th]
\begin{center}
\includegraphics[width=0.6\textwidth]{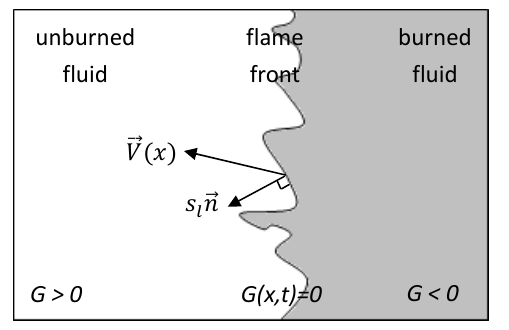}
\end{center}
\caption{\label{fig:Geq}G-equation model.}
\end{figure}

For a given moment $t$, let the flame front be the zero level set of a function
$G(x,t)$. Then the burnt region is $G(x,t) < 0$, the unburnt region is $G(x,t) > 0$,  
the normal direction of their interface pointing from the burnt region to the
unburnt region is $DG/|DG|$, and the normal velocity is $-G_t/|DG|$.
The motion law becomes the so called $G$-equation, a well-known model in turbulent combustion
\cite{W85,Pet00}:
\be
G_t + V(x)\cdot DG + s_l |DG|=0. \label{ge1}
\ee
See Fig.\,\ref{fig:Geq}. Chemical kinetics and Lewis number effects are all included in the laminar speed $s_l$ which is
provided by a user.  In general, $s_l$ might not be a constant.  Throughout this paper,  we consider the basic G-equation by setting $s_l=1$.  The most general G-equation model where $s_l$ incorporates both curvature and strain rate effects was formally introduced by Williams in \cite{W85} in 1985. An earlier form of
G-equation and the associated level set approach appeared in Markstein’s work \cite{M1951} in 1951 and \cite{M1964} in 1964.  G-equation also serves as one of the main computational examples in
the systematic development of level-set method by Osher-Sethian \cite{OS1988}.

{\bf The prediction of the turbulent flame speed is a fundamental problem in turbulent combustion theory} \cite{W85,R,Pet00}.  Roughly speaking, the turbulent flame speed is the averaged flame propagation speed under the influence of strong flows.  Under the G-equation model, the turbulent flame speed $s_{\rm T}(p)$ along a given unit direction $p\in \Rset^n$ is given by
\be\label{limitdef}
s_{\rm T}(p) := \lim_{t\to +\infty}-\frac{G(x,t;p)}{t},
\ee
where the convergence on the right hand side holds locally uniformly for all $x\in \Rset ^n$, and, importantly, is independent of $x$.  Here $G(x,t;p)$ is the unique viscosity solution of
equation (\ref{ge1}) with initial data $G(x,0;p)=p\cdot x$. For  simplicity of notations, we often use $G(x,t)$ without writing explicitly the dependence on $p$. In combustion literature, the turbulent flame speed is often defined and experimentally measured by the ratio between areas of the wrinkled flame front and its projection to the plane $p\cdot x=0$.  This is consistent with (\ref{limitdef}) under the G-equation model \cite{KAW1988,XYR2023}.

  The existence of the limit (\ref{limitdef})  has been independently established in  \cite{CNS}  and \cite{XY_10} for Lipschitz continuous, periodic, and near incompressible flows in all dimensions.  In  homogenization theory, $s_{\rm T}(p)$ is called the ``effective Hamiltonian", which is the unique number such that the following cell problem
 \be\label{eq:cell}
 |p+Dw|+V(x)\cdot (p+Dw)=s_{\rm T}(p) \quad \text{in $\Rset^n$}
 \ee
 has approximate periodic viscosity solutions. The function $s_{\rm T}(p): \Rset^n \to \Rset$ is known to be convex and positive homogeneous of degree 1 (i.e., $s_{\rm T}(\lambda p)=\lambda s_{\rm T}(p)$ for all $\lambda\geq 0$).  See the survey paper \cite{XYR2023} for review of homogenization theory and viscosity solutions.

Now we change $V$ to $AV$ for a constant $A>0$ that is called flow intensity (or stirring intensity), and let $s_{\rm T}(p,A)$ be the corresponding turbulent flame speed.  A practically significant and mathematically interesting question is to determine the growth law of  $s_{\rm T}(p,A)$ as $A\to +\infty$.  The increase of $A$ is often achieved by mechanically rotating fluid within the combustion chamber \cite{HZ1994}.    By applying  the renormalization theory to the inviscid G-equation model,  Yakhot \cite{Yak} formally derived  the following growth law
$$
s_{\rm T}(p,A)=O\left(\frac{A}{\sqrt{\log A}}\right),
$$
assuming that the flow $V$ is statistically isotropic.   The above law has been  considered as a benchmark in combustion literature. 
 
 A natural question is whether this $O\left({A/\sqrt{\log A}}\right)$ growth law can be rigorously established for a class of mathematically interesting and physically meaningful  flows $V$. The first thought is to look at isotropic stochastic flows. However, these types of flows are usually only H\"older continuous in spatial variables, where the well-posedness of the equation (\ref{ge1}) is not clear. The pure transport equation $u_t+V(x)\cdot Du=0$  was known to have multiple solutions with given initial data when $V$ is merely H\"older continuous. See \cite{CL1983,DEIJ2022} for non-uniqueness  examples. To avoid the well-posedness issue and unknown existence of $s_{\rm T}(p)$, we consider Lipschitz continuous $V$  throughout this paper.  A slightly weaker notion is log-Lipschitz continuity, see Remark \ref{rmk:log-lip}.  Physically, Lipschitz continuous flows
sit in the Batchelor (smooth) regime of turbulence flows \cite{BGK_98,Son_99}.
 
 As the first step,  we focus on two dimensional periodic Lipschitz continuous flows with mean zero.  Below are three concrete examples. I and II   appear often  in math and physics literature \cite{CG}.

 \medskip
 
 \begin{itemize}
 
 \item[{\upshape (I)}] \emph{Cellular flows}.  A prototypical example is 
 \be\label{eq:example}
 V(x)=(-H_{x_2},H_{x_1}) \quad \text{ for $H(x)=\sin (2\pi x_1)\sin (2\pi x_2)$}.
 \ee
Here $x=(x_1,x_2)$. The associated growth law was known to be
\begin{equation*}
s_{\rm T}(p,A)=O\left(\frac{A}{\log A}\right);
\end{equation*}
see \cite{CTVV2003, NXY, O2001, XY2013} for reference, and see \cite{XY2013} for the sharp constant. 

\medskip

\item[{\upshape (II)}] \emph{Flows with open channels}. Two representative examples are
\begin{itemize}
\item[(a)] \emph{Shear flow}: $V(x)=(v(x_2),0)$ for a periodic Lipschitz continuous non-constant function $v: \Rset\to \Rset$ that has mean zero.

\item[(b)] \emph{Cat's-eye flows}: $V(x)=(-H_{x_2},H_{x_1})$ where $H$ depends on a parameter $\delta \in (0,1)$ and is given by 
 $$
 H_\delta(x)=\sin (2\pi x_1)\sin (2\pi x_2)+\delta \cos (2\pi x_1)\cos (2\pi x_2).
 $$
\end{itemize}
The associated growth law was known to be 
\be\label{eq:open}
s_{\rm T}(p,A)=
\begin{cases}
O(1)  \quad &\text{if } p\cdot p_0=0,\\
O(A)\quad &\text{if } p\cdot p_0\not =0, 
\end{cases}
\ee
where $p_0$ is a unit vector that is parallel to the direction of the open channel. For example, $p_0 = (1,0)$ for the shear flow above, and $p_0 = \frac{1}{\sqrt{2}}(1,1)$ for the cat's eye flow above. See \cite{XY2014} for more detailed descriptions.

\begin{figure}[th]
\begin{center}
\includegraphics[width=0.25\textwidth]{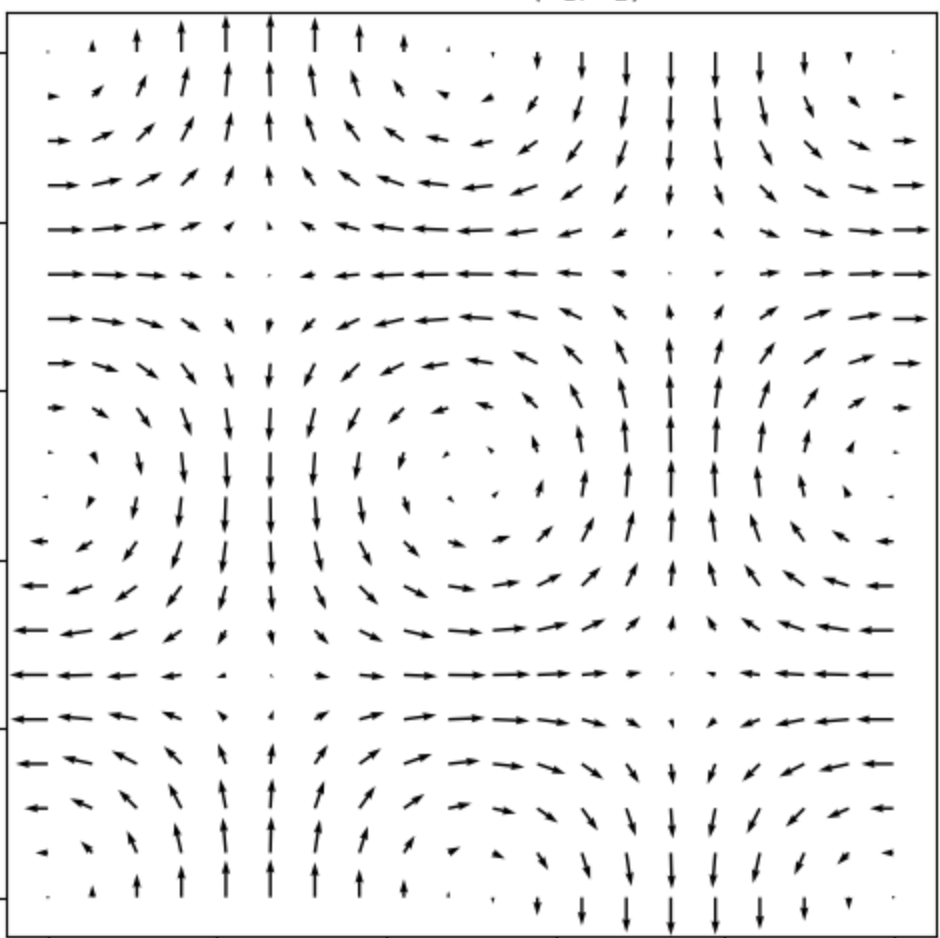}\qquad
\includegraphics[width=0.25\textwidth]{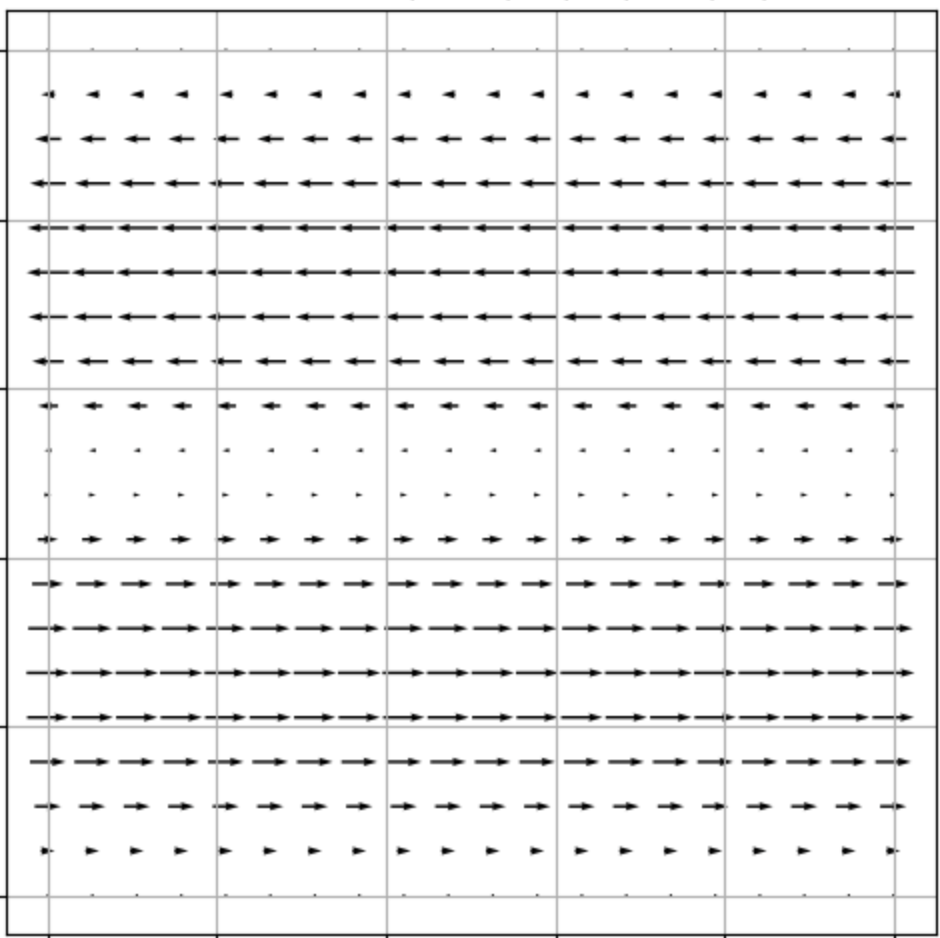}\qquad 
\includegraphics[width=0.25\textwidth]{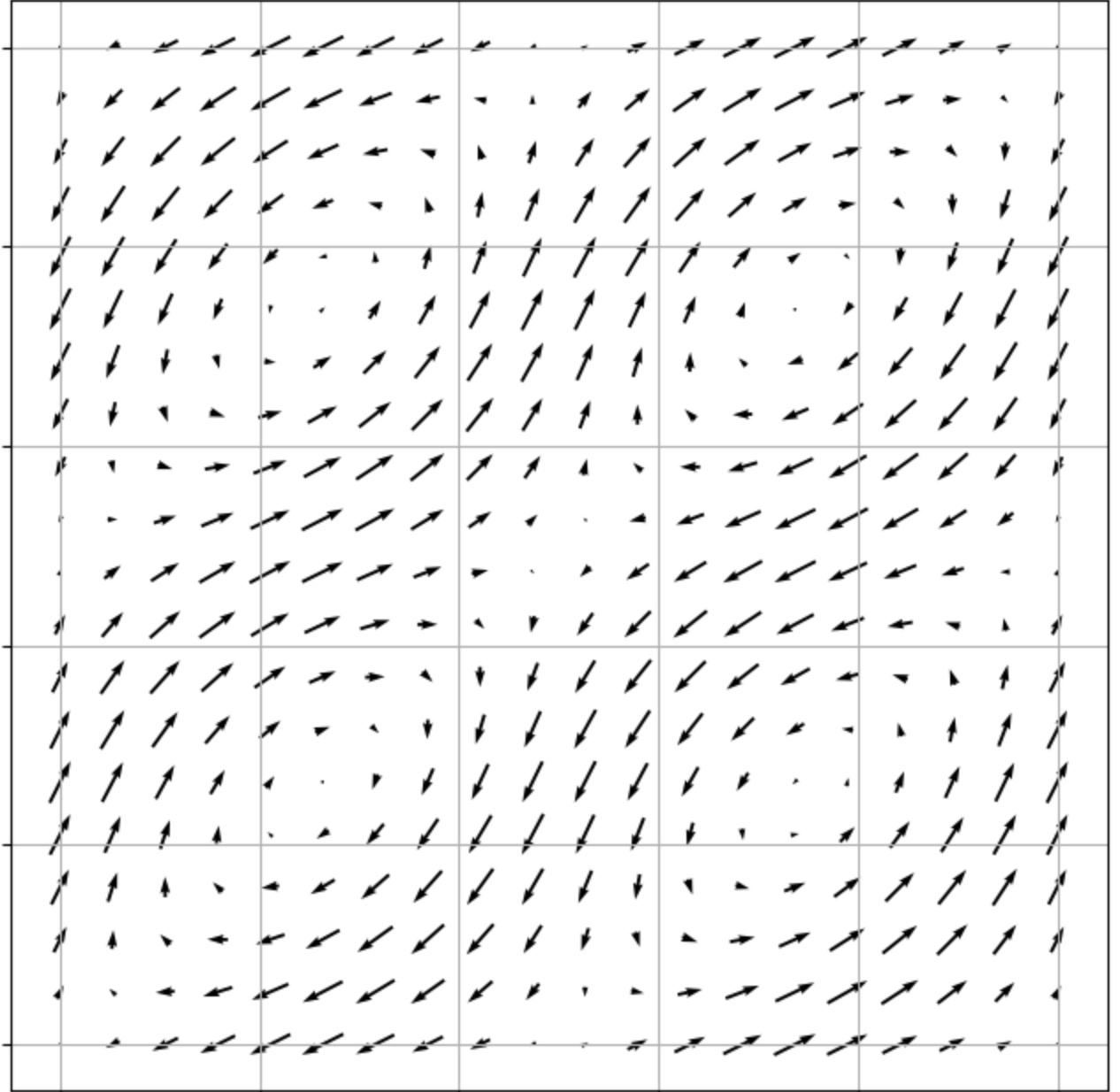}
\end{center}
\caption{\label{fig:threeflow}Cellular flow, shear flow and cat's-eye flow ($\delta=1/2$)}
\end{figure}

 \item[{\upshape (III)}] {\it Two-scale flow.}  An observation made in \cite{CTVV2003} was that introducing additional scales does not significantly affect the growth law. However, this may not always hold true. For example, consider  \( V(x) = (-H_{x_2}, H_{x_1}) \), where \( H \) is defined as
\be\label{eq:twoscale}
H(x) = \sin(2\pi x_1) \sin(2\pi x_2) + 0.3 \cos(6\pi x_1) \cos(6\pi x_2),
\ee
and compare it with the cellular flow \eqref{eq:example}. The inclusion of the smaller-scale term creates a flow with an open channel structure along the direction \( p_0 = \frac{1}{\sqrt{2}}(1,1) \). See Figure \ref{fig:twoscaleflow} below for level curves of $H$, which has both global and local extrema. 

 \begin{figure}[th]
\begin{center}
\includegraphics[width=0.5\textwidth]{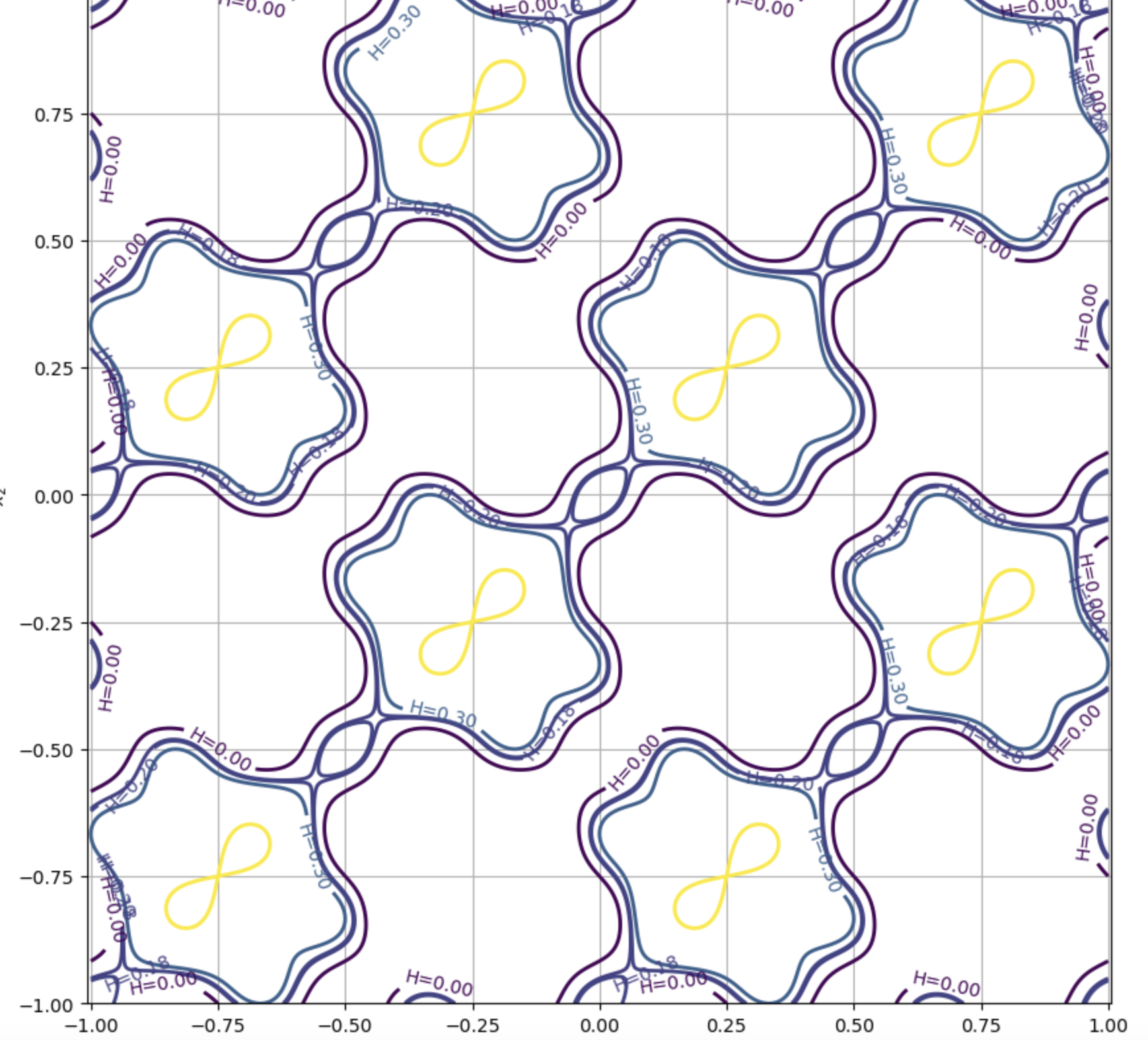}
\end{center}
\caption{\label{fig:twoscaleflow}Level set of $H$ from (\ref{eq:twoscale})}
\end{figure}
\end{itemize}

To differentiate an $O\left(\frac{A}{ \sqrt{\log A}}\right)$ rate from an $O\left(\frac{A}{\log A}\right)$ one requires  large values of $A$, which is very challenging to carry out numerically or empirically.  Since the cellular flow has only one scale,  it is discussed in \cite{CTVV2003} via numerical computations whether adding more scales could lead to Yakhot's $O\left(\frac{A}{\sqrt{\log A}}\right)$ law. In this paper, we will show that {\it there is no intermediate growth law between $O(A)$ and $O\left(\frac{A}{\log A}\right)$ for two-dimensional Lipschitz continuous incompressible flows with bounded swirl sizes}.  

Let us first introduce some notations and concepts before stating the main theorem.  Denote by $\mathbb{T}^n=\Rset^n/\Zset^n$ the $n$-dimensional flat torus and $W^{1,\infty}(\mathbb{T}^n, \Rset^n)$ the set of Lipschitz continuous $\Zset^n$-periodic functions from $\Rset^n$ to $\Rset^n$.   
 
Given a flow $V\in W^{1,\infty}(\mathbb{T}^2, \Rset^2)$, a curve  $\xi\in C^{1,1}(\Rset, \Rset^2)$ is said to be  an  orbit of the flow $V$ if
 $$
 \dot \xi=V(\xi(t)) \quad \text{for all $t\in \Rset$}.
 $$
 We say that two orbits $\xi_1$ and $\xi_2$ are the same if one is a time translation of the other, i.e.  $\xi_1(t)=\xi_2(t+t_0)$ for all $t\in \Rset$ and a fixed $t_0\in \Rset$.  Moreover,  a subset $E \subseteq \Rset^2$ is called flow invariant if for all $x\in E$, $\xi(\cdot;x)(\Rset)\subseteq E$. Here, $\xi(\cdot;x)$ is the orbit subject to $\xi(0;x)=x$. We call the image of an orbit $\xi$ a streamline of $V$.

An orbit $\xi$ of $V$ is called closed if $\xi (0)=\xi(T)$ for some $T \in (0,\infty)$ and $T$  is called a period of the closed orbit. The set of stagnation points (i.e., critical points) $\Gamma$ of $V$ is defined by 
\begin{equation}
\label{eq:stagset}
\Gamma=\{x\in \Rset^2|\   V(x)=0\}.
\end{equation}
 We say that an orbit $\xi$ is asymptotic to $\Gamma$ if 
 \be\label{eq:asym}
 \lim_{t\to \pm \infty}d(\xi(t), \Gamma)=0.
 \ee
 
By \emph{swirl} of the flow $V$, we mean the image of a closed orbit $\xi$ of $V$. Throughout this paper,  for technical convenience, we assume that the sizes of swirls of $V$  have an upper bound; that is, there exists $M>0$ such that for any closed orbit $\xi$, 
 \be\label{eq:boundedswirl}
\max_{0\leq t_1,t_2\leq \Rset} |\xi(t_1)-\xi(t_2)|\leq M.
 \ee
  This assumption matches what happens in all real situations of turbulent combustion because the size of the swirl cannot go beyond a multiple of the diameter of the combustion chamber.  Also,  we assume that $V$ has mean zero, i.e., 
 \be\label{eq:meanzero}
 \int_{\Rset^2}V(x)\,dx=0,
 \ee
 which is consistent with the isotropic assumption in \cite{Yak}.  From a mathematical point of view,  the mean zero assumption ensures that $s_{\rm T}(p)$ is positive in all directions. Indeed, taking integration on both sides of  the cell problem (\ref{eq:cell}) leads to (note $\nabla \cdot V = 0$ due to incompressibility) 
 $$
 s_{\rm T}(p)\geq \, |p|,
 $$
which shows the enhancement of flame propagation speed due to the presence of the flow $V$. The following is our main result.

 \begin{theo}\label{theo:main} Assume that $V\in W^{1,\infty}(\mathbb{T}^2,  \Rset^2)$ is incompressible, mean zero and its swirls have uniformly bounded sizes. That is, $\mathrm{div}(V)=0$, a.e., \eqref{eq:boundedswirl} and \eqref{eq:meanzero} hold.  Then, 
either (1) or (2) in the following holds:
 
\begin{itemize}
\item[\upshape(1)] There exists a unit vector $p_0\in \Rset^2$ such that
$$
\begin{cases}
\displaystyle \lim_{A\to +\infty} \frac{s_{\rm T}(p,A)}{A} =c_p \quad &\text{if } \; p\cdot p_0\not=0,\\
\displaystyle \sup_{A\geq 0} \; s_{\rm T}(p,A) <\infty \quad &\text{if } \; p\cdot p_0=0.
\end{cases}
$$
Here $c_p$ is a positive constant depending only on $p$ and $V$. 

\item[\upshape(2)] For every unit vector $p\in \Rset^2$, 
 $$
 \limsup_{A\to +\infty} \frac{s_{\rm T}(p,A)\log A}{A}<\infty. 
$$
\end{itemize}
\end{theo}

\begin{rmk}  Steady two dimensional incompressible periodic flows considered in the above theorem are integrable. Turbulent flows are non-integrable.  It is therefore interesting to explore whether the presence of Lagrangian chaos could enhance the propagation speed and lead to different  growth laws. For representative three dimensional flows like the Arnold-Beltrami-Childress (ABC) flow and the Kolmogorov flow,  $O(A)$ growth law is known in certain situations \cite{XYZ, KLX}.  The linear growth is equivalent to the residue front as laminar flame speed $s_l$ tends to zero. See Remark \ref{rmk:residue}.  In sections 4, we also demonstrate that the growth law associated with the unsteady and mixing cellular flow is not faster than $O(A/\log A)$.  A simple way to generate a growth law like $O(A/\sqrt{\log A})$ is to slightly reduce the regularity of the flow $V$ from Lipschitz continuity which is essentially the Batchelor (smooth) regime of turbulence flow \cite{BGK_98,Son_99}.
See Remark \ref{rmk:log-lip} for details. 
A physical example from the class of rough (H\"older continuous) stochastic flows \cite{BGK_98} awaits to be 
found to support the growth law $O(A/\sqrt{\log A})$.

\end{rmk}

$\bullet$ {\bf Sketch of the proof}. According to the control interpretation of solutions to convex Hamilton-Jacobi equations \cite{Evans_98}, the solution $G(x,t)$ of equation (\ref{ge1}) with a Lipschitz initial data $G(x,0)=g(x)$ has a representation formula
\be\label{eq:control-general}
-G(x,t)=\sup_{\gamma \in \Sigma_t}\left(-g(\gamma (t))\right),
\ee
where $\Sigma_t$ is the set of all
 Lipschitz continuous curves $\gamma$ (control paths) defined on $[0,t]$   satisfying $\gamma(0)=x$ and $|\dot \gamma(\cdot)+V(\gamma(\cdot))|\leq 1$, a.e. in $[0,t]$. In particular,  the solution $G(x,t;p)$ with initial data $g=p\cdot x$ is given by 
\be\label{eq:control}
-G(x,t;p)=\sup_{\gamma \in \Sigma_t}\left(-p\cdot \gamma (t)\right),
\ee

First, we quickly recall how to obtain the bound $O(A/\log A)$ for the cellular flow in (\ref{eq:example}) and refer to \cite{NXY, O2001, XY2013} for the details. The method relies on the available simple explicit formulation.  For $\eta\in \Sigma_t$, write $\eta(s)=(x_1(s),x_2(s))$.  Then we have
  $$
 | \dot x_1(s)|\leq 2\pi A |\sin (2\pi x_1(s))|+1 \quad \text{for a.e. $s\in [0,t]$},
 $$
 and, hence, for $A\geq 2$ it takes at least 
 $$
 \int_{0}^{1} \frac{1}{1+2\pi A |\sin (2\pi x_1)|}\,dx_1=O\left(\frac{\log A}{ A}\right) 
 $$
{amount of time} for $\eta$ to pass the stripe bounded by two lines $x_1=k$ and $x_1=(k+1)$ for each $k\in \Zset$. Due to \eqref{limitdef} and \eqref{eq:control}, this immediately leads to the upper bound
 $$
 s_{\rm T}(e_1,A)\leq O(A/\log A).
 $$
A similar argument holds for the $x_2$ direction. By choosing a suitable control path $\gamma$ that travels close to the separatrices (level curves of the critical value $H=0$), we can also get $s_{\rm T}(p,A)\geq O(A/\log A)$.  To obtain sharp constants would require a more delicate analysis \cite{XY2013}.

\begin{rmk}\label{rmk:log-lip} The above proof strategy and calculations can be easily extended to handle $V$ with lower regularity. In particular,  the growth law  $O(A/\sqrt{\log A})$ can be obtained mathematically if we lower the Lipschitz regularity of the vector field to $\frac12$-log-Lipschitz although its physical meaning is not clear.  Here we say that a function $f$ is $\alpha$-log-Lipschitz continuous for some $\alpha\geq 0$ if 
$$
|f(x)-f(y)|\leq C|x-y|({-\log |x-y|})^{\alpha} \quad \text{for all $x,y\in \Rset^2$ with $|x-y|\leq 1$}.
$$
We refer to \cite{MR2014} and reference therein for more information on this class of functions. 

For example, we can consider the stream function
$$
H(x_1,x_2)=\sin (2\pi x_1)\sin (2\pi x_2)\cdot
\sqrt{-\log \left(\frac{\sin^2 (2\pi x_1)+\sin^2 (2\pi x_2)}{4}\right)}.
$$
and $V=(-H_{x_2},H_{x_1})$. Simple computations show that  for $i=1,2$
$$
|H_{x_i}|\leq C|\sin 2\pi x_{i'}|\sqrt{-\log |\sin 2\pi x_{i'}}|, 
$$
where $i' = \{1,2\}\setminus \{i\}$ is the complementary index of $i\in \{1,2\}$. Also, we get
$$
V(x_1,0)=\left(-2\pi\sin (2\pi x_1)\sqrt{-2\log \left({\sin (2\pi x_1)\over 2}\right)},\,  0\right)
$$
and
$$
V(0,x_2)=\left(0, \, 2\pi\sin (2\pi x_2)\sqrt{-2\log \left({\sin (2\pi x_2)\over 2}\right)}\right). 
$$
Hence $V$  is $C^{\infty}$ on $\Rset^2\backslash {\Zset^2\over 2}$ and is $1\over 2$-log-Lipschitz in $\Rset^2$. For $\alpha\in [0,1]$,  if $V$ is $\alpha$-log-Lipschitz continuous,  it is not hard to show that $G(x,t)$ given by the control formulation (\ref{eq:control-general}) is H\"older continuous and is the unique solution to (\ref{ge1}). The existence of $s_T$ can be established using the same method in \cite{XY_10}.  We leave this to the interested readers to explore. Then the $O(A/\sqrt{\log A})$ growth law follows from
$$
\int_{0}^{1}{1\over 1+As\log (-s)}\,ds=O(A/\sqrt{\log A}).
$$

 \end{rmk}

For more general two-dimensional incompressible flow $V$, the key steps to establish the $O(A/\log A)$ growth rate are as follows.
 
\medskip

{\bf Step 1:} \emph{Analyze the cell structures of the streamlines of $V$ away from the set $\Gamma$ of stagnation points defined by \eqref{eq:stagset}.}
For two-dimensional incompressible flow $V$, we can always find a scalar field $H$ so that $V = (-H_{x_2},H_{x_1})$. Hence, the set $\Gamma$ of stagnation points of $V$ are precisely the critical points of $H$. Henceforth, we also refer to the points in $\Gamma$ as critical points. If all critical points are non-degenerate, i.e. $\mathrm{det}(DV(x)) \ne 0$ for all $x\in \Gamma$, the structure of the streamlines is well understood \cite{A1991}:  it consists of finitely many cells bounded by separatrices of $H$. The main novelty of this paper is that we do not assume the non-degeneracy of critical points. Consequently, topologically complicated situations might arise.   For example,   the number of cells and scales might be infinite within one period.  Hence, we need to properly define cells and consider those maximal cells. 
\medskip

{\bf Step 2:} \emph{Establish the result in item (2) of Theorem \ref{theo:main} assuming that there is no non-contractible periodic orbit.} Indeed, according to \cite{XY2014}, case (1) of Theorem \ref{theo:main}, i.e. a dichotomy between $O(A)$ and $O(1)$ growth rates, occurs if and only if there exist non-contractible periodic orbits of $V$; see \eqref{eq:noncon} for the definition. Therefore, we rule out non-contractible periodic orbits. Using the control formulation (\ref{eq:control}), we only need to show that any control path takes at least $O({\log A\over A})$ time to pass a stripe with fixed width $2M+1$ where $M$ is the bound of swirl size in (\ref{eq:boundedswirl}).

There are two main ingredients to achieve the goal: First, structural results from Step 1 ensure that such a path has to travel through a maximal cell within the stripe by connecting two points on the boundary of the cell that are not on the same orbit. Second, Corollary \ref{cor:crosscell} asserts then that for such a path, either the travel time is no less than $\log A\over A$ or  $\gamma$ must contain a point $x_0$ such that 
\begin{equation*}
|V(x_0)|\leq 3K_0\left({K_0\log A\over A}\right)^{1\over 3}
\end{equation*}
where $K_0=\|V\|_{W^{1,\infty}(\Rset^2)}$. Therefore,
$$
|\dot \gamma|\leq AK_0|\gamma(s)-x_0|+C(\log A)^{1\over 3}A^{2\over 3}+1.
$$
Then it takes the path at least 
$$
\int_{0}^{1}{1\over 1+ C+C(\log A)^{1\over 3}A^{2\over 3}+ACr}\,dr {= O\left({\log A\over A}\right)}
$$
amount of time to escape $B_1(x_0)$. The uniform bound of swirl sizes (\ref{eq:boundedswirl}) is only used to control the size of a cell. 

 \medskip
 
\noindent{\bf More notations:} Given a set $D\subseteq \Rset^n$,  $\overline {D}$ and $D^{\circ}$ represent the closure and the interior of $D$ respectively. For two sets, $A\subset B$ means that $A$ is a proper subset of $B$. For a curve $\gamma:\Rset\to \Rset^n$ and any subset $I\subseteq \Rset$, $\gamma(I)=\{\gamma(t)|\ t\in I\}$. 

\medskip
 
\noindent{\bf Outline of the paper.}  In section 2, we study the structure of streamlines of $V$.  In Section 3 we prove Theorem \ref{theo:main} following the aforementioned plan. Examples of two dimensional unsteady cellular flows and three dimensional steady flows are discussed in section 4.

 \section{Structures of the streamlines}
Throughout Sections 3 and 4, we assume that $V$ satisfies the assumptions of Theorem \ref{theo:main}. In this section,  we study the structure of two dimensional periodic incompressible flows without assuming the non-degeneracy of critical points. Some contents might be well known to experts, but are still presented for readers' convenience.  Some results are intuitively clear, but we take effort to prove them rigorously.
  
 For  $V\in W^{1,\infty}(\mathbb{T}^n, \Rset^n)$,   an orbit $\xi:[0,T]\to \Rset$ is called a \emph{non-contractible periodic orbit} if  for some $T>0$, 
 \be\label{eq:noncon}
 \xi(T)-\xi(0)\in \Zset^n\backslash\{0\}.
 \ee
  $T$ is called a  period of $\xi$.  Note that $\xi$ is a non-contractible  periodic orbit if and only if it is a non-contractible closed orbit when it is projected to the flat torus $\mathbb{T}^n$. A point $x\in \Rset^n$ is called a \emph{periodic} point if it is either on a closed orbit or a non-contractible periodic orbit.  More generally,  a point $x\in \Rset^2$ is called \emph{recurrent} if there exist an orbit $\xi(\cdot;x)$ starting from $x$ and a sequence $T_m\to +\infty$ such that 
 $$
 \lim_{m\to +\infty}d(\xi(T_m;x), x+\Zset^n)=0.
 $$

 Owing to Poincar\'e recurrence theorem,  almost every $x\in \Rset^n$ is a recurrent point under the incompressible flow $\dot \xi=V(\xi)$. 
 
 When $n=2$, due to the incompressibility and mean zero assumptions of $V$, there is a scalar field $H\in C^{1,1}(\mathbb{T}^2,\Rset)$, henceforth called the \emph{stream function}, such that  for $x=(x_1,x_2)$,
 $$
 V(x)=D^{\perp}H(x)=(-H_{x_2}, H_{x_1}).
 $$
Apparently, $H$ is constant along any orbit of $V$.   Note that given $x\in \Rset^2\backslash \Gamma$,  there exists a neighbourhood $U_x$  of $x$ such that for any $y\in U_x$, $H(y)=H(x)$ if and only if $y$ and $x$ are on the same orbit. More detailed discussions will be given later.   Hence every recurrent point  is a periodic point.  Note that this coincidence in general is not valid in higher dimensions when $n\geq 3$.   
 
 Hereafter, we assume that $n=2$. 
 
  \begin{lem}\label{lem:threecases} Any orbit $\xi$ belongs to one of the following categories:
  
  (1) $\xi$ is  a closed orbit; 
  
  (2) $\xi$ is  a non-contractible periodic orbit;
  
  (3) $\xi$  is asymptotic  to $\Gamma$, i.e., (\ref{eq:asym}) holds. 
 \end{lem}
 
\begin{proof}   Suppose that $\xi$ is neither a  closed orbit  nor a non-contractible  periodic orbit.  The goal is to establish (3).  We argue by contradiction.  If $\xi$ is not asymptotic  to $\Gamma$, then there exists $x_0\in \Rset^2\backslash \Gamma$ and a sequence $T_m\to +\infty$ such that 
 $$
 d(\xi(T_m),x_0+\Zset^2)=0.
 $$
Then $H(\xi(t))=H(x_0)$ for all $t\in  \Rset$. Since $DH(x_0)\not= 0$, similar to the previous discussion about equivalence between recurrent points and periodic points in two-dimensional space, we must have that $x_0$ is on $\xi(\Rset)+\Zset^2$ and is a periodic point. Accordingly,  $\xi$ is either a closed orbit or a non-contractible periodic orbit. This is a contraction.
\end{proof}

 \begin{figure}
  \begin{center}
\includegraphics[scale=0.3]{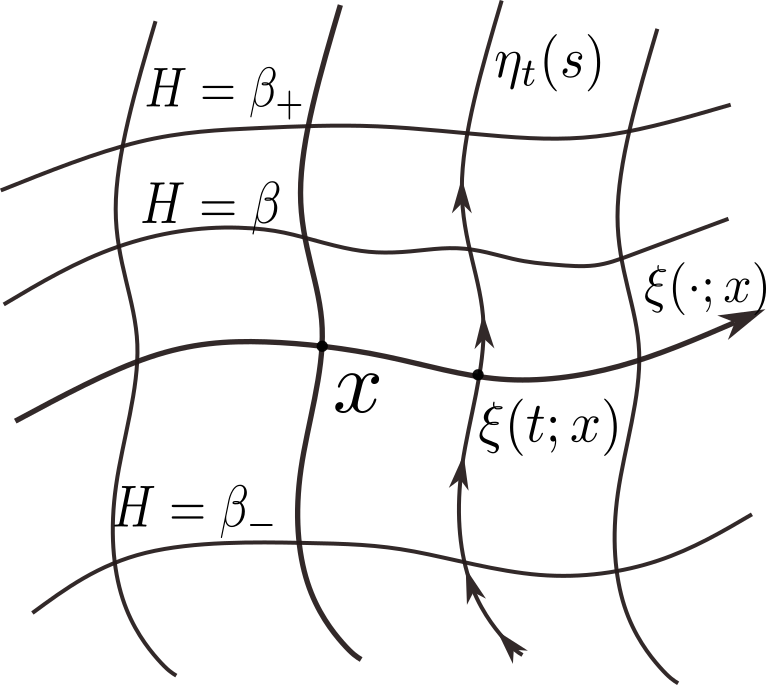}
\includegraphics[scale=0.3]{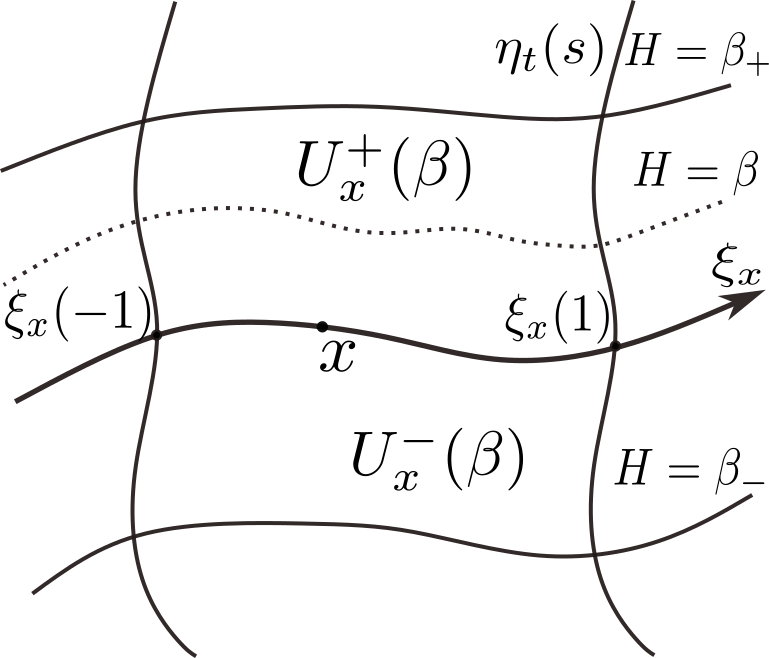}
\end{center}
 \caption{\label{fig:localframe} New coordinate system near $x$.}
\end{figure}
 
Given any $x\notin \Gamma$, we introduce a new coordinate system near $x$ that will be convenient for our purposes. Let $\xi = \xi(\cdot;x): \Rset\to \Rset^2\backslash \Gamma$ be the orbit with $\xi(0)=x$.  For $t\in \Rset$, let $\eta_t(s):\Rset\to \Rset^2$ be the solution to 
 $$
 \begin{cases}
\frac{d}{ds} \eta_t(s) = DH(\eta_t(s)) \quad \text{for $s\in \Rset$}\\[3mm]
 \eta_t(0)=\xi(t).
 \end{cases}
 $$
See Fig.\,\ref{fig:localframe}. Clearly, $H(\eta_t(s))$ is strictly increasing with respect to $s$ for fixed $t$. Let
 $$
 \Omega_x := \{(t,\beta)|\ \lim_{s\to -\infty}H(\eta_t(s)) < \beta<\lim_{s\to +\infty}H(\eta_t(s)).
 $$
 and let $\wj{s_\beta}\in \Rset$ be the unique number such that
$$
 H(\eta_t(\wj{s_\beta}))=\beta.
 $$
 Define the map $\Phi(t,\beta):\Omega_x\to \Rset^2$ as
 \be\label{eq:orbitfunction}
 \Phi(t,\beta)=\eta_t(\wj{s_\beta})
 \ee
 Note that $\Phi$ is a local homeomorphism to its image.  
 Define:
 $$
 \beta_+=\min\{H(\eta_t(1))|\ t\in [-1,1]\} \quad \mathrm{and} \quad  \beta_-=\max\{H(\eta_t(-1))|\ t\in [-1,1]\}.
 $$
  For all $t\in [-1,1]$,  $H(\xi(t;x))=H(x)$ and
 $$
 H(\eta_t(-1))\leq \beta_{-}< H(x)< \beta_+\leq H(\eta_t(1))
  $$
  
Here for clarity of notations, we omit the dependence of  $\beta_+$, $\beta_-$  and $\Phi$ on $x$.  Write three open sets
\begin{equation}
\label{eq:2equations}
 \begin{aligned}
 U_x &= \Phi((-1,1)\times (\beta_-,\beta_+)),\\
U_x^{-}(\beta) &= \Phi((-1,1)\times (\beta_-,\ \beta)),\\
 U_x^{+}(\beta) &= \Phi((-1,1)\times (\beta,\ \beta_+)).
\end{aligned}
\end{equation}

If $\xi:[0,T]\to \Rset^2$ is a closed orbit subject to $\xi(0)=\xi(T)$,  we write $R_\xi$ as the closed region bounded by $\xi$. Clearly, for two closed orbits $\xi_1$ and $\xi_2$ that have different images, 
 \be\label{eq:notinclude}
 R_{\xi_1}\cap R_{\xi_2}\not=\emptyset  \quad \Rightarrow \quad   \text{$R_{\xi_1}\subset R_{\xi_2}^\circ$ or $R_{\xi_2} \subset R_{\xi_1}^\circ$}.
 \ee

 \begin{lem} \label{lem:twoparts} Given $x\notin \Gamma$ and $y\in U_x$, let $\xi_y$ denote the orbit $\xi(\cdot;y)$ with $\xi_y(0)=y$,  and assume further that $\xi_y$ is a closed orbit. Then one and only one of the following holds (See Figure \ref{fig:U_x}).

\begin{itemize}
\item[(1)] $U_x^{+}(H(y))\subset R_{\xi_y}$ and $U_x^{-}(H(y))\cap R_{\xi_y}=\emptyset$;

\item[(2)] $U_x^{-}(H(y))\subset R_{\xi_y}$ and $U_x^{+}(H(y))\cap R_{\xi_y}=\emptyset$.
\end{itemize}
\end{lem}

\begin{proof}  For simplicity of notations, we write $\beta=H(y)$ and $W^\pm=U_x^{\pm}(\beta)$.  Since $H(z)\not=\beta$ for all $z\in W^{\pm}$,   $W^{\pm}\cap \xi_y(\Rset)=\emptyset$. Apparently, $\Phi([-1,1]\times \{\beta\})\subset \xi_y(\Rset)=\partial R_{\xi_y}$ and 
$$
(W^{+}\cup W^{-})\cap R_{\xi_y}\not=\emptyset.
$$
There are two cases: If $W^{+}\cap R_{\xi_y}\not=\emptyset$, since $W^{+}\cap \xi_y(\Rset)=\emptyset$ and $W^{+}$ is connected, we must have $W^{+}\subset R_{\xi_y}$. Similarly, if $W^{-}\cap R_{\xi_y}\not=\emptyset$, we must have $W^{-}\subset R_{\xi_y}$.
\medskip

Finally we observe that, if the two cases happen at the same time, we get $(W^{+}\cup W^{-})\subset R_{\xi_y}$, and hence a neighborhood of $y$ will be contained in $R_{\xi_y}$, which would contradict the fact that $y\in \partial R_{\xi_y}$. Therefore, either Case 1 and hence (1) hold, or Case 2 and hence (2) hold; see Fig.\,\ref{fig:U_x}.
\end{proof}

\begin{figure}
\begin{center}
\includegraphics[width=0.4\textwidth]{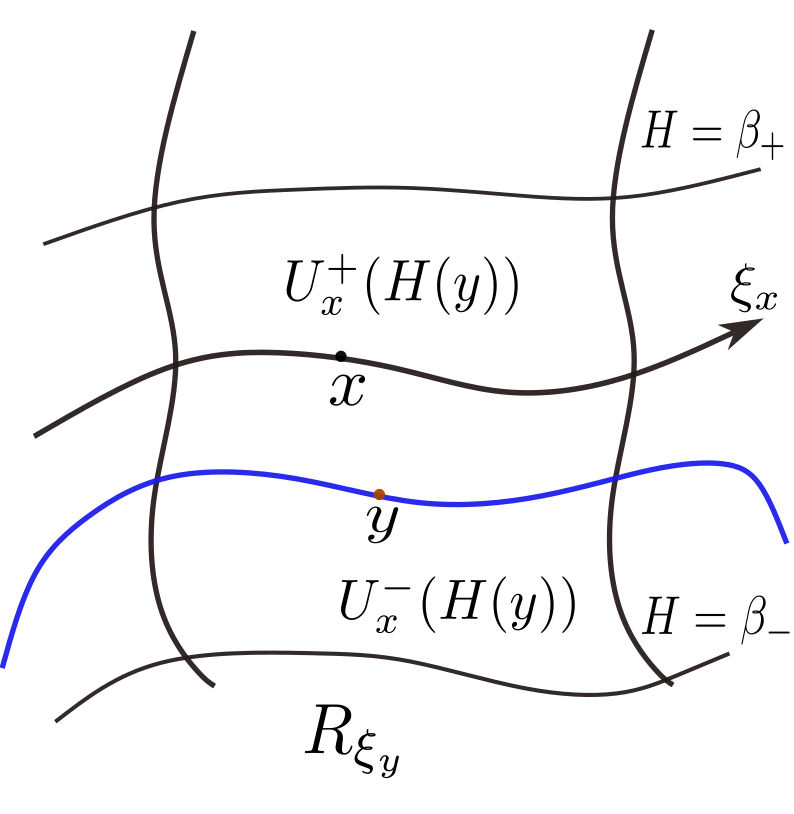}
\includegraphics[width=0.4\textwidth]{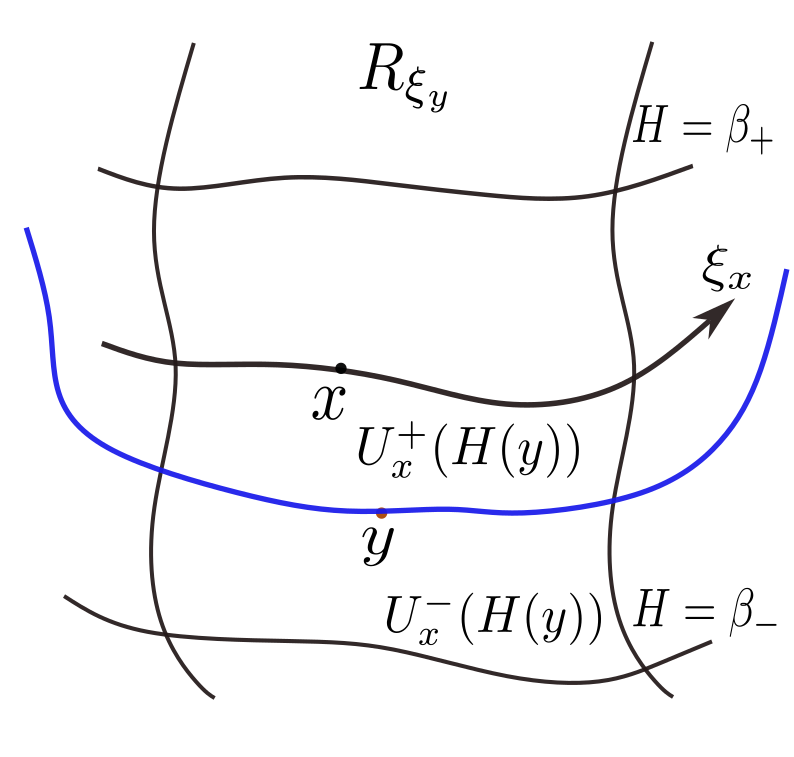}
\end{center}
\caption{\label{fig:U_x} Possible relations between $R_{\xi_y}$ (region enclosed by the orbit $\xi_y$ which is partially shown in blue color) and $U^\pm_x(H(y))$.}
\end{figure}

Moreover, we have that the neighborhood of a closed  orbit/non-contractible periodic orbit is foliated by closed orbits/non-contractible periodic orbit. Specifically speaking, let $\xi:  [0,T]\to \Rset^2$ be a closed orbit with $\xi(0)=x\notin \Gamma$ and $T>0$ be its minimum period. Let
$$
\alpha_+=\min\{H(\eta_t(1))|\ t\in [0,T]\} \quad \mathrm{and} \quad  \alpha_-=\max\{H(\eta_t(-1))|\ t\in [0,T]\}.
$$
Then the following results hold (and similar conclusions hold for non-contractible periodic orbits). 
 \begin{lem}\label{lem:foliation} The map   $\Phi(t,\beta)$  defined in (\ref{eq:orbitfunction})  associated with $\xi$ is a homeomorphism from $[0,T)\times [\alpha_-, \alpha_+]$ to its image.  In particular, for fixed $\beta\in [\alpha_-, \alpha_+]$, $\Phi(t,\beta):[0,T]\to \Rset^2$ is a closed curve of $\{H=\beta\}$. 
 \end{lem}

Next, we introduce the definition of cells that play a key role in describing the streamline structure of the flow $V$; see Fig.,\ref{fig:Cell} for an illustration.
  
\begin{defin}\label{def:cell}   A closed set $S$ is called a cell if there exist a sequence of  closed orbits $\{\xi_m\}_{m\geq 1}$ such that $\{R_{\xi_{m}}\}_{m\geq 1}$ is a strictly increasing sequence and 
 \begin{equation}
 \label{eq:def:cell}
 S=\overline{\cup_{m\geq 1}R_{\xi_{m}}} \quad \text{and} \quad \partial S\cap \Gamma\not=\emptyset.
 \end{equation}
 Here $\Gamma=\{x\in \Rset^2|\ V(x)=0\}.$   Also a cell $S$ is called maximal if there does not exist another cell $S'$ such that $S$ is a proper subset of $S'$. 
 \end{defin}
 
We would like to point out that the topology of $\partial S$ near $\Gamma$ could be complicated since $V$ might have degenerate critical points.   
 \begin{figure}
  \begin{center}
\includegraphics[width=0.36\textwidth]{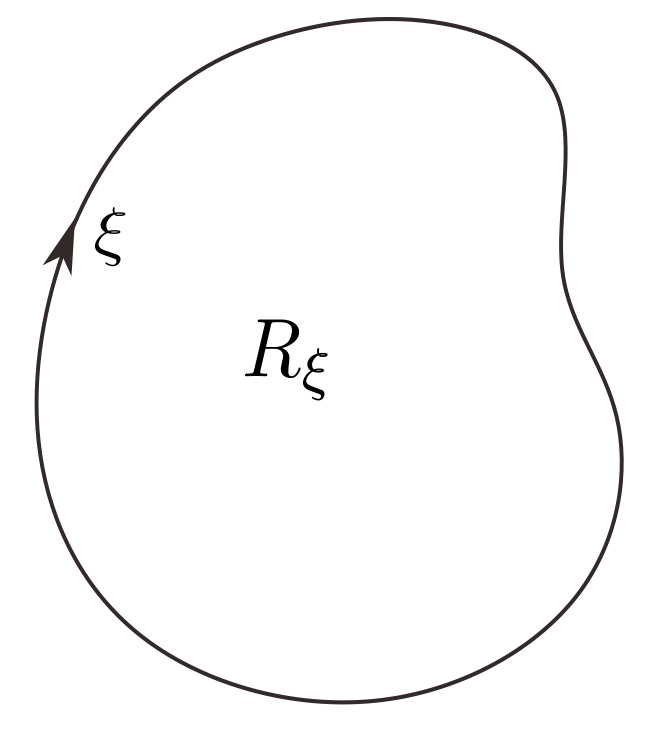} 
\ \ \ 
\includegraphics[width=0.36\textwidth]{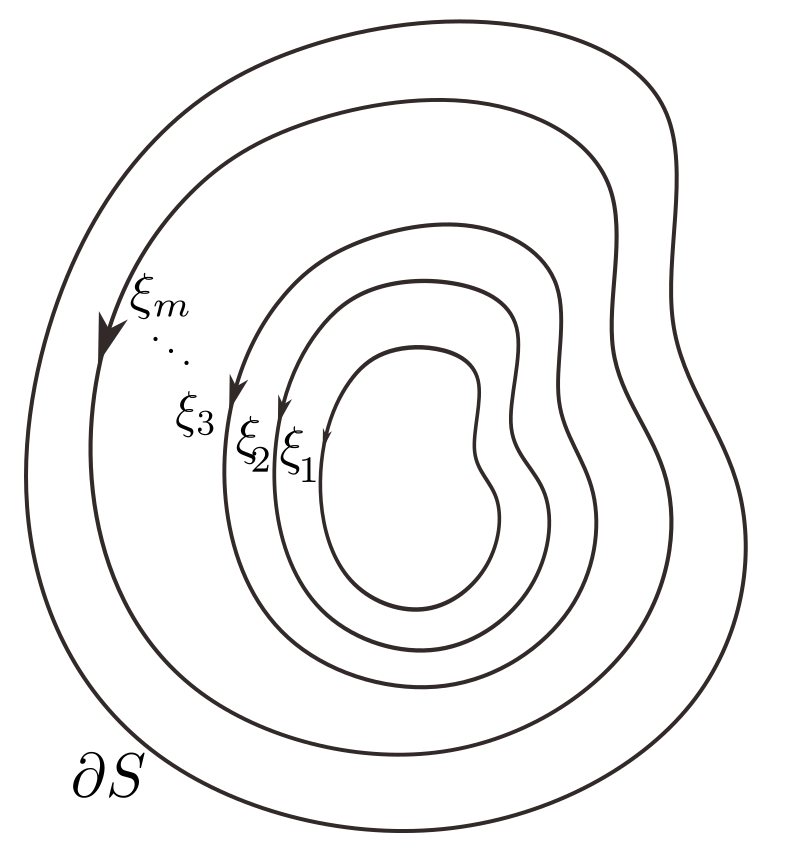}
\end{center}
\caption{\label{fig:Cell} Picture of $R_\xi$ and a cell $S$}
\end{figure}

Below are several basic topological properties of a cell. 
 
\begin{lem}\label{lem:property} Let $S$ be a cell and $\{\xi_m\}_{m\geq 1}$ be a sequence satisfying \eqref{eq:def:cell}.   Then the following results hold.
 \begin{itemize}
 \item[(1)]  $S$ is bounded, closed, connected and $S=\overline{S^{\circ}}$; moreover, $S$, $S^{\circ}$ and $\partial S$ are all flow invariant.
 
 \item[(2)] Let $\mathcal{L}_{\{\xi_m\}_{m\geq1}}=\{x\in \Rset^2|\ x=\lim_{m\to +\infty}x_m \quad \text{for some $x_m\in \xi_m(\Rset)$}\}$. Then
  \be\label{eq:characterization}
  S^{\circ}\backslash \Gamma=(\cup_{m\geq 1}R_{\xi_{m}})\backslash \Gamma \quad \mathrm{and} \quad \partial S\backslash \Gamma=\mathcal{L}_{\{\xi_m\}_{m\geq1}}\backslash{\Gamma}
  \ee
   and $\partial S\subset \mathcal{L}_{\{\xi_m\}_{m\geq1}}$. Moreover, as a consequence, there is a constant $c \in \Rset$ such that
  \be\label{eq:constantbd}
  H(x)\equiv c \quad  \text{on $\partial S$}.
 \ee

  \item[(3)]  For every $x\in \partial S\backslash \Gamma$, the orbit $\xi_x = \xi(\cdot;x)$ is asymptotic to $\Gamma$ and  $\xi_x(\Rset)\subset \partial S$.

  \item[(4)] For any closed orbit $\xi$, if $R_\xi\cap S\not=\emptyset$, then either $R_\xi\subset S^{\circ}$ or $S\subset R_\xi^{\circ}$.
 
 \end{itemize}
 \end{lem}
 
 \begin{proof}  Throughout the proof, $\xi_x:\Rset\to \Rset^2$ represents the orbit of $V$ satisfying $\xi_x(0)=x$.
 
 \smallskip

 \emph{Proof of (1)}. It is obvious that $S$ is bounded, closed and connected. To see that $S^\circ$ is dense in $S$, we notice that the sequence $\{R_{\xi_m}\}_{m\geq 1}$ is strictly increasing and hence $R_{\xi_m} \subset R_{\xi_{m+1}}^\circ$, so the union $\cup_{m\geq 1}R_{\xi_{m}}$ is an open set, and hence
 $$
 \cup_{m\geq 1}R_{\xi_{m}}^\circ = \cup_{m\geq 1}R_{\xi_{m}}\subset S^{\circ}.
 $$
This shows $S=\overline{S^{\circ}}$.  We would like to point out that $S^{\circ}$ might  be larger than  $\cup_{m\geq 1}R_{\xi_{m}}$ due to the possible degeneracy of critical points.  It is clear that each $R_{\xi_m}$ is flow invariant. Hence $S$ is flow invariant. Since the \emph{flow determined by $V$} is a global homeomorphism, i.e., for fixed $t$, $x\mapsto \xi_x(t)$ is a homeomorphism of $\Rset^2$,  we deduce that $S^{\circ}$ and $\partial S$ are also flow invariant. 
 
 \medskip

\emph{Proof of (2)}. We first establish two simple facts. 
 
 \smallskip
 
\emph{Claim 1}: $x\in \mathcal{L}_{\{\xi_m\}_{m\geq1}}$ if and only if there exists a subsequence $\{\xi_{m_k}\}_{k\geq 1}$ of $\{\xi_m\}_{m\ge 1}$ such that we can find $x_{m_k}\in \xi_{m_k}(\Rset)$ and $\lim_{k\to +\infty}x_{m_k}=x$ .  
 \medskip
 
The \emph{only if} part is obvious. To prove the \emph{if} part, for $k\geq 1$ and $m\in [m_k, m_{k+1}]$,  we can choose 
$$
x_m\in \xi_m(\Rset)\cap \{sx_{m_k}+(1-s)x_{m+1}|\ s\in [0,1]\}.
$$
Then it is easy to see that $x=\lim_{m\to +\infty}x_m$. 
 
 \medskip
 
 \emph{Claim 2}:
 \be\label{eq:claim}
 S=(\cup_{m\geq 1}{R_{\xi_{m}}})\cup \mathcal{L}_{\{\xi_m\}_{m\geq1}}. 
 \ee
By the definition  of $S$, $S$ contains the set on the right; it suffices to prove the other direction of inclusion. Suppose $x\in S$ but $x\notin \mathcal{L}_{\{\xi_m\}_{m\geq1}}$, then by \emph{Claim 1} there exists $r>0$ and $N \in \Nset$ such that 
 $$
 B_r(x)\cap \left(\cup_{m\geq N}\xi_m(\Rset)\right)=\emptyset. 
 $$
 By definition of $S$, $B_r(x)\cap (\cup_{m\geq 1}R_{\xi_{m}})\not=\emptyset$ and $R_{\xi_m}$ is increasing in $m$, so we can choose $m_1\geq N$ such that 
 $$
 B_r(x)\cap R_{\xi_{m_1}}\not=\emptyset.
 $$
 Since $B_r(x)$ is connected and $B_r(x)\cap \xi_{m_1}(\Rset)=\emptyset$, we must have $B_r(x)\subset R_{\xi_{m_1}}$. This establishes \emph{Claim 2}.  
 
\medskip


As a result, $\partial S\subset \mathcal{L}_{\{\xi_m\}_{m\geq1}}$. For each $m\ge 1$, $H\rvert_{\xi_m(\Rset)}$ is a constant denoted below by $c_m$. For any $x\in \mathcal{L}_{\{\xi_m\}}$,  let $x_m \in \xi_m(\Rset)$ so that $x_m\to x$. We get 
$$
H(x) = \lim_{m\to \infty} H(x_m) = \lim_{m\to \infty} c_m =: c.
$$
The last limit is independent of $x$, so $H \equiv c$ on $\mathcal{L}_{\{\xi_m\}}$ and hence on $\partial S$.
 \medskip

 Now fix $x\in   \mathcal{L}_{\{\xi_m\}_{m\geq1}}\backslash \Gamma$. Clearly, $x\notin \cup_{m\geq 1}R_{\xi_m}$. Without loss of generality,  up to a subsequence if necessary, we may assume that $c_m\leq c$ for all $m\geq 1$. Hence
 $$
 U_x^{+}(c)\cap \xi_{m}(\Rset)=\emptyset, \quad \forall m\ge 1.
 $$
See (\ref{eq:2equations}) for the definition of $U_x^{+}(c)$. Since $U_x^{+}(c)$ is connected and $x\notin \cup_{m\geq 1}R_{\xi_m}$,  we must have that 
 $$
 U_x^{+}(c)\cap \left(\cup_{m\geq 1}R_{\xi_m}\right)=\emptyset.
 $$
 Otherwise, we will have $U_x^{+}(c)\subset R_{\xi_m'}^{\circ}$ for some $m'\in \Nset$. This implies that $x\in {\overline U_x^{+}(c)}\subset R_{\xi_m'}$, which contradicts to the choice of $x$.  Hence  $U_x^{+}(c)\cap S=\emptyset$ and $x\in \partial S$.  Thus 
 $$
 \mathcal{L}_{\{\xi_m\}_{m\geq1}}\backslash \Gamma\subset (\partial S)\backslash \Gamma. 
 $$
This establishes one equality in \eqref{eq:characterization}, the other equality follows from \eqref{eq:claim}.  

\medskip

\emph{Proof of (3)}. For $x\in \partial S\backslash \Gamma$, if $\xi_x$ is not asymptotic to $\Gamma$, then by Lemma \ref{lem:threecases}, $\xi_x$ is a closed orbit. Since $\partial S$ is flow invariant, $\partial R_{\xi_x} = \xi_x(\Rset) \subset \partial S$ and is a connected component of $\partial S$. In view of \eqref{eq:characterization} we may find a sequence of closed orbits $\{\xi_m\}_{m\ge 1}$ so that $\{R_{\xi_m}\}$ is strictly increasing, satisfying
$$
S = \overline{\cup_{m\ge 1} R_{\xi_m}}, \qquad \text{and } \quad \lim_{m\to +\infty}\xi_m(0)=x.
$$
Then owing to Lemma \ref{lem:foliation},  $S=R_{\xi_x}$. This is a contradiction to the requirement $\partial S\cap \Gamma\not=\emptyset$ for $S$ being a cell. This shows, for all $x\in \partial S\setminus \Gamma$, $\xi_x$ is asymptotic to $\Gamma$.  

\medskip

\emph{Proof of (4)}.  Assume that $S\cap R_\xi\not=\emptyset$. owing to (3), $\partial S\cap \xi(\Rset)=\emptyset$. Since $\xi(\Rset)$ is connected, there are two cases.
\medskip

Case 1: $\xi(\Rset)\subset \Rset^2\backslash S$.  Since $S$ is connected and $S\cap R_\xi\not=\emptyset$, we must have 
$$
S\subset R_\xi^{\circ}.
$$

\medskip

Case 2: $\xi(\Rset)\subset S^\circ$. Since $\xi(\Rset)\cap \Gamma=\emptyset$, owing to (\ref{eq:characterization}),  there exits $m_0\in \Nset$ such that 
$$
\xi(\Rset)\subset R_{\xi_{m_0}},
$$
which implies that $R_\xi\subset S^{\circ}$.
\end{proof}


\begin{lem}\label{lem:onedirection} Let $S$ be a cell, and $c= H\rvert_{\partial S}$. Then, one and only one of the following holds:

\begin{itemize}
\item[(1)] $\cup_{x\in \partial S\backslash \Gamma}U_x^{+}(c)\subset S^\circ$ and  $\cup_{x\in \partial S\backslash \Gamma}U_x^{-}(c)\subset \Rset^2\backslash S$;
\item[(2)] $\cup_{x\in \partial S\backslash \Gamma}U_x^{-}(c)\subset S^\circ$ and  $\cup_{x\in \partial S\backslash \Gamma}U_x^{+}(c)\subset \Rset^2\backslash S$.
\end{itemize}
\end{lem}

\begin{proof}  By Definition \ref{def:cell} there is a strictly increasing sequence of $\{R_{\xi_m}\}$ such that 
$S=\overline{\cup_{m\geq 1}R_{\xi_m}}$. Note that $H\rvert_{\xi_m(\Rset)}$, restricted to each of the closed orbits $\xi_m$, is a constant $c_m$. By choosing a subsequence if necessary, it suffices to look at the following two cases. 

\emph{Case 1: $c_m\geq c$ for $m\geq 1$}. We establish (1).  It suffices to show that for  any fixed $x\in \partial S\backslash \Gamma$, $U^+_x(c) \subset S^\circ$ and $U_x^-(c)\subset \Rset^2\setminus S$.  

In fact, without loss of generality,  we may assume that 
$$
\{\xi_m(0)\}_{m\geq 1}\subset U_x^{+}(c) \quad \mathrm{and} \quad \lim_{m\to +\infty}\xi_m(0)=x.
$$
Upon a subsequence if necessary,  we may also assume that $c_m$ is strictly decreasing.

\smallskip

Let $\Phi$ be the map defined in (\ref{eq:orbitfunction}).  Then for $m\geq 2$,  we have $c_m <c_1$, $\Phi([-1,1]\times \{c_1\})\subset \xi_1(\Rset)\subset R_{\xi_1}$, $R_{\xi_1}\subset R_{\xi_m}$, and hence
$$
\emptyset \neq \Phi([-1,1]\times \{c_1\})\subset (R_{\xi_m}\cap U_x^+(c_m)).
$$
On the other hand, we have $\xi_m(\Rset)\cap U^+_x(c_m) = \emptyset$. Since $U_x^+(c_m)$ is connected, we must have that 
$$
U^+_x(c_m) \subset R_{\xi_m}.
$$
By taking union of all $m\in \Nset$,  we derive that
$$
U_x^{+}(c)=\cup_{m\geq 1}U^+_x(c_m)\subset \cup_{m\geq 1}R_{\xi_m} \subseteq S^\circ.
$$

To show $U_x^-(c)$ is outside $S$, we prove the following claim.
\smallskip

\emph{Claim}: $U_x^{-}(c) \cap R_{\xi_m}=\emptyset$,  for all $m\geq 1$.
\smallskip

Suppose the Claim does not hold, so $U_x^{-}(c) \cap R_{\xi_{m_0}}\not=\emptyset$ for some $m_0\ge 1$. Since $c<c_{m_0}$, $U_x^{-}(c)\cap \xi_{m_0}(\Rset)=\emptyset$. Note again that $\xi_{m_0}(\Rset) = \partial R_{\xi_{m_0}}$. Because $U_x^-(c)$ is connected, we must have 
$$
 U_x^{-}(c)\subset R_{\xi_{m_0}}^{\circ}\subset S^{\circ}. 
$$
Hence there exists $r>0$, such that $B_r(x)\subset S$. This contradicts to $x\in \partial S$.  This proves the claim, and it follows that $U_x^-(c) \subset \Rset^2\setminus S$. Thus (1) holds in Case 1.

\medskip

\emph{Case 2:  $c_m\leq c$ for $m\geq 1$}. By exchanging the roles of $U_x^+$ and $U_x^-$, the same proof above leads to (2).   
\end{proof}

 Also, we have that 
 \begin{cor}\label{cor:twodirections} Suppose that $V$ does not admit any non-contractible periodic orbit. Then, for any $x\in \Rset^2\backslash{\Gamma}$ on an orbit $\xi$ that is asymptotic to $\Gamma$,  there exists a cell $S$ such that $x\in \partial S$. 
 \end{cor}
 
 \begin{figure}
\begin{center}\includegraphics[width=.7\textwidth]{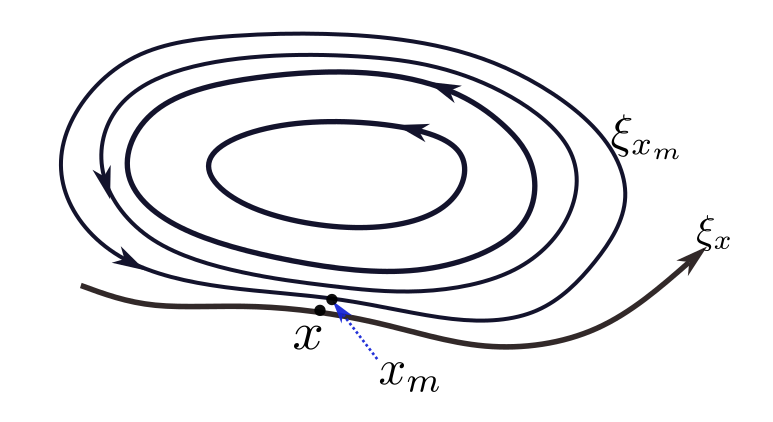}
\end{center}
\caption{\label{fig:asymCell}Construction of a cell close to an orbit that is asymptotic to $\Gamma$}
\end{figure}
 
\begin{proof} Let $c=H(x)$ and $U_x$, $U_x^\pm(\cdot)$ be defined by \eqref{eq:2equations}.  By the Poincar\'e recurrence theorem,  we can choose two sequence of points $\{x_{m}\}_{m\geq 1}\subset U_x^{+}(c)$ and $\{\tilde x_{m}\}_{m\geq 1}\subset U_x^{-}(c)$ such that all of the following hold:
 
 \begin{itemize}
 \item[(i)] $\lim_{m\to +\infty}x_{m}=\lim_{m\to +\infty}\tilde x_{m}=x$; 
 
\item[(ii)] For each $m \ge 1$, the orbits $\xi_m := \xi(\cdot;x_m)$ and $\tilde \xi_m = \xi(\cdot;\tilde x_m)$ are closed orbits with $\xi_m(0)=x_m$ and $\tilde \xi_m(0)=\tilde x_m$.

\item[(iii)] Let $c_m := H(x_m) = H\rvert_{\xi_m}$, and $\tilde c_m := H(x_m) = H\rvert_{\tilde \xi_m}$. Then $\{c_m\}_{m\ge 1}$ is strictly decreasing and the sequence  $\{\tilde c_{m} \}_{m\ge 1}$ is strictly increasing;
 
 \end{itemize}

{\bf Claim:} For all $m,n\in \Nset$,
$$
x\notin R_{\xi_m}\cap R_{\tilde \xi_n}.
$$
If not, suppose that $x\in R_{\xi_m}\cap R_{\tilde \xi_n}$. By (\ref{eq:notinclude}),  without loss of generality, we assume  that $R_{\xi_m}\subset R_{\tilde \xi_n}^\circ$. Since $x\in U_{x}^{-}(c_m)\cap U_{x}^{+}(\tilde c_n)$,   Lemma \ref{lem:twoparts} implies that 
$$
U_{x}^{-}(c_m)\subset R_{\xi_m} \quad \text{and} \quad U_{x}^{+}(\tilde c_n)\subset R_{\tilde \xi_n}.
$$
Accordingly,  $\tilde \xi_n(0)\in U_{x}^{-}(c_m)\subset R_{\tilde \xi_n}^\circ$, which is absurd. Hence the above claim holds. 
  So, upon choosing a subsequence, without loss of generality, we may assume that 
  $$
  x\notin \cup_{m\geq 1} R_{\xi_m}.
  $$
Since $x\in U_x^-(c_m)$ for all $m\geq 1$,  Lemma \ref{lem:twoparts} implies that 
\be\label{eq:separate}
U_{x}^{+}(c_m)\subset R_{\xi_m} \quad \text{and} \quad U_{x}^{-}(c_m)\cap R_{\xi_m}=\emptyset.
\ee
On the one hand, $c_{m+1}<c_m$ implies $\xi_m(0) \in U_x^+(c_{m+1})$, and hence the first relation above implies $R_{\xi_m}\cap R_{\xi_{m+1}}\not=\emptyset$. 
On the other hand, $c_{m+1}<c_m$ also shows $\xi_{m+1}(0)\in U_{x}^{-}(c_m)$, and hence the second relation above implies $\xi_{m+1}(0)\in R_{\xi_{m+1}}\backslash R_{\xi_{m}}$. Thus by \eqref{eq:notinclude} we must have $R_{\xi_m}\subset R_{\xi_{m+1}}$ for all $m\ge 1$, which shows that $\{R_{\xi_m}\}_{m\ge 1}$ is strictly increasing. Let
 $$
 S=\overline {\cup_{m\geq 1}R_{\xi_{m}}}.
 $$
Clearly, $x\in S$. As we have checked before, $S$ and $\partial S$ are flow invariant. Since $ U_x^{-}(c)\subset U_x^{-}(c_m)$ for all $m\geq 1$, (\ref{eq:separate}) implies that
$$
 U_x^{-}(c)\cap \left(\cup_{m\geq 1}R_{\xi_{m}}\right)=\emptyset.
 $$
Therefore, $S\cap U_x^{-}(c)=\emptyset$.  So $x\in \partial S$.  By flow invariance of $\partial S$, we have $\xi_x(\Rset) \subset \partial S$, and since $\xi_x$ is asymptotic to $\Gamma$, we have $\partial S\cap \Gamma \neq \emptyset$.  As a result, $S$ is a cell.
\end{proof}

\begin{cor}\label{cor:con} Let $S_1$ and $S_2$ be two different cells. Then exactly one of the following holds:

 \begin{itemize}
 \item[(1)] $S_1\cap S_2=\partial S_1\cap \partial S_2$; or

 \item[(2)] There exists a closed orbit $\xi$  such that
 $$
 S_1\subset R_{\xi} \subset S_2^{\circ} \quad \text{or} \quad S_2\subset R_{\xi}\subset S_1^{\circ}.
 $$
 
\end{itemize}

\end{cor}

\begin{proof}  Assume that (1) does not hold. Without loss of generality, we may assume that
$$
S_1\cap S_2^{\circ}\not=\emptyset. 
$$
 In view of Definition \ref{def:cell}, we can choose two sequences of closed orbits, denoted respectively by $\{\xi_{m,1}\}_{m\ge 1}$ and $\{\xi_{m,2}\}_{m\ge 1}$, so that the corresponding sequence of regions enclosed by them, i.e., $\{R_{\xi_{m,1}}\}_{m\geq 1}$ and $\{R_{\xi_{m,2}}\}_{m\geq 1}$, are strictly increasing and satisfy
 $$
 S_1=\overline {\cup_{m\geq 1}R_{\xi_{m,1}}} \quad \mathrm{and} \quad S_2=\overline {\cup_{m\geq 1}R_{\xi_{m,2}}}.
 $$
 Then there must exist $m_1\in \Nset$ such that for $m\geq m_1$,
 $$
 R_{\xi_{m,1}}\cap S_2\not=\emptyset.
 $$
 Hence owing to (4) in Lemma \ref{lem:property}, there are two cases.

 \medskip

 Case 1:  $S_2\subset R_{\xi_{m',1}}$ for some $m'\geq m_1$. Then our conclusion holds.

 \medskip

 Case 2:  $\cup_{m\geq m_1}R_{\xi_{m,1}}\subset S_2$.  Then
 $$
 S_1\subseteq S_2. 
 $$
So $S_1\cap \left(\cup_{m\geq m_1}R_{\xi_{m,2}}\right)\not=\emptyset$.  Again, thanks to (4) in Lemma \ref{lem:property}, either there exists $m_2$ such that $S_1\subset R_{\xi_{m_2,2}}$ and we get our conclusion, or if otherwise, we must have 
$$
\cup_{m\geq 1}R_{\xi_{m,2}}\subset S_1.
$$
This leads to $S_2\subseteq S_1$ and, therefore,  $S_2=S_1$,  which contradicts to the assumption that $S_1$ and $S_2$ are two different cells.
\end{proof}

\begin{lem}\label{lem:maximum} Suppose that $V$ does not admit any non-contractible periodic orbits. Then the following holds.
\begin{itemize}
\item[(1)] For any $x\in \Rset^2\backslash \Gamma$, there exists a maximal cell $S$ such that $x\in S$.   
\item[(2)] For any maximal cell $S$ and any closed orbit $\xi$ of $V$, then either $R_\xi\subset S^{\circ}$ or $R_\xi\cap S=\emptyset$.  
\end{itemize}
\end{lem}

\begin{proof}  \emph{Proof of (1)}. Fix $x\in \Rset^2\backslash \Gamma$,  we consider two settings.

\smallskip

\emph{Case 1:  There exists a closed orbit $\xi$ such that $x\in R_\xi$}. Let $J_x$ be the non-empty set of all closed orbits $\eta$ subject to $x\in R_{\eta}$. Consider the set
\be\label{eq:construction1}
\mathcal{O}_x := \cup_{\eta\in J_x}R_{\eta}.
\ee
 Due to the local foliation near a closed orbit (Lemma \ref{lem:foliation}),   for any $\eta\in J_x$, there exists $\eta'$ such that $R_\eta\subset R_{\eta'}^{\circ}$. 
Hence, we see that $\mathcal{O}_x$ is open.
\smallskip

\emph{Claim:  $S_x =\overline {\mathcal{O}_x}$ is the unique maximal  cell containing $x$}.   

\smallskip
\emph{Step 1: We show $S$ is a cell.} Due to \eqref{eq:notinclude}, for $\eta_1, \eta_2\in J_x$, either $R_{\eta_1}\subset R_{\eta_2}$ or $R_{\eta_2}\subset R_{\eta_1}$. Accordingly, there exists a strictly increasing sequence $\{R_{\xi_m}\}_{m\geq 1}$ such that $\xi_m\in J_x$ and
 $$
 \mathcal{O}_x =\cup_{m\geq 1}R_{\xi_m}^{\circ}=\cup_{m\geq 1}R_{\xi_m}. 
 $$

 It remains to show that $\partial S_x\cap \Gamma\not=\emptyset$.  In fact, for any $y\in \partial S_x$, if $y\not\in \Gamma$, then since $S_x$ is flow invariant and the flow is a homeomorphism,  $\xi_x(\Rset)\subset \partial S_x$.  Similar to the proof of (3) of Lemma \ref{lem:property},  $\xi_y$ has to be asymptotic to $\Gamma$ and $\partial S_x\cap \Gamma\not=\emptyset$.  Otherwise, $\xi_y$ is a closed orbit and $S_x\subset R_{\xi_y}$, which contradicts to the definition of $S_x$. 
 
 \smallskip

 \emph{Step 2: We verify the maximality of $S$ and its uniqueness.}  Assume that $S'$ is a cell such that $x\in S'$. Since $x\in S_x^{\circ}$, due to Corollary \ref{cor:con} and the definition of $S_x$, we must have $S'\subseteq S_x$.
 \smallskip

\emph{Case 2: $x \not\in R_{\eta}$ for any periodic orbit $\eta$}. First by Lemma \ref{lem:threecases} and the non-existence of non-contractible periodic orbit, the orbit $\xi_x = \xi(\cdot;x)$ is asymptotic to $\Gamma$.  By Corollary \ref{cor:twodirections}, there is a cell $S$ such that $x\in \partial S$. 

\smallskip

\emph{Maximality of $S$}: Assume that $S'$ is a cell such that $S\subset S'$.  According to  Corollary \ref{cor:con},  there exists a closed orbit $\xi$ such that
$$
S\subset R_\xi\subset S'.
$$
This implies that $x\in R_\xi$, which contradicts to the choice of $x$.

\medskip
 
 \emph{Proof of (2)}. Assume that $R_\xi\cap S\not=\emptyset$ so we can find $x\in R_\xi\cap S$. Then $x\not\in \Gamma$, and by the construction of the unique maximal cell $S_x$ containing $x$ in the proof of (1), we must have $S=S_x$ and $R_\xi\subset \mathcal{O}_x\subset S_x^{\circ}$.
 \end{proof}

We say that two cells $S_1$ and $S_2$ of $V$ are adjacent if $\partial S_1\cap \partial S_2 \neq \emptyset$. Similarly to the proof of Corollary \ref{cor:twodirections},  we have the following corollary.  
 
 \begin{cor}\label{cor:adjacent} Suppose that $V$ does not admit any non-contractible periodic orbit.  Suppose that $S_1$ is a maximal cell and $x\in \partial S_1\backslash \Gamma$.  Then there exists a different maximal cell $S_2$ such that $x\in \partial S_2$.  In particular, $S_1$ and $S_2$ are adjacent cells. 
 \end{cor}
 \begin{proof}  Let $c = H(x) = H\rvert_{\partial S_1}$, and let $\xi_x$ be the orbit $\xi(\cdot;x)$. Owing to Lemma \ref{lem:onedirection},  we may assume without loss of generality that 
 $$
 U_{x}^{-}(c)\subset S_1^{\circ} \quad \text{and} \quad U_{x}^{+}(c)\subset \Rset^2\backslash S_1.
 $$
Choose a sequence $\{\xi_m\}_{m\ge 1}$ of closed orbits so that
\begin{equation*}
 x_m := \xi_m(0)\in U_x^{+}(c) \to x, \; c_m := H\rvert_{\xi_m(\Rset)} \searrow c, \; \quad \text{ as } \, m\to \infty.
 \end{equation*}
 Above, we may assume that $\{c_m\}$ is strictly decreasing. 

 \medskip
 
 Since $\xi_m(0)\notin S_1$,  by Lemma \ref{lem:maximum},  we must have that $S_1\cap R_{\xi_m}=\emptyset$ for all $m\geq 1$. Hence $U_{x}^{-}(c)\cap R_{\xi_m}=\emptyset$ for all $m\geq 1$. Also,  note that $x\in S_1\cap U_x^{-}(c_m)$.  So,   Lemma \ref{lem:twoparts} leads to
 $$
 U_x^{+}(c_m)\subset R_{\xi_m} \quad \mathrm{and} \quad  U_x^{-}(c_m)\cap R_{\xi_m} =\emptyset, \quad \forall m\ge 1.
 $$
Since $U^+_x(c_m) \subset U^+_x(c_{m+1})$, we see that $R_{\xi_m}\cap R_{\xi_{m+1}}\ne \emptyset$. Also, note that $\xi_{m+1}(0)\in R_{\xi_{m+1}}\cap U_x^{-}(c_m)\subseteq  R_{\xi_{m+1}}\backslash R_{\xi_m}$.  In view of \eqref{eq:notinclude}, we have $R_{\xi_m}\subset R_{\xi_{m+1}}$ for all $m\geq 1$.  Let
 $$
 S_2=\overline {\cup_{m\geq 1}R_{\xi_m}}.
 $$
 Clearly, $S_2$ is flow invariant and $x\in S_2$.  Since $S_2\cap U_{x}^{-}(c)=\emptyset$, $x\in \partial S_2$.  This implies that $\xi_x(\Rset)\in \partial S_2$ since the flow is a homeomorphism. Note that $\xi_x$ is asymptotic to $\Gamma$ by (3) of Lemma \ref{lem:property}. So $\partial S_2\cap \Gamma\not=\emptyset$. This shows that $S_2$ is a cell and $x\in \partial S_1\cap \partial S_2$.   Since $S_1\cap R_{\xi_m}=\emptyset$,  $S_1\ne S_2$. Suppose $S_3$ is another cell such that $S_2\subset S_3$.  Then by Corollary \ref{cor:con}, there exists a closed orbit $\xi$ such that
 $$
 S_2\subset R_\xi\subset S_3.
 $$
 In particular, $x\in R_\xi\cap S_1$. By the maximality of $S_1$ and (2) of Lemma \ref{lem:maximum}, $R_\xi\subset S_1$, which leads to $S_2\subset S_1$, which is impossible by the construction of $S_2$. Hence, $S_2$ is also maximal. This completes the proof.
 \end{proof}

\begin{lem}\label{lem:bound} Assume that  $x,y\in \Rset^2$ satisfy  $H(x)=H(y)$ and $y\notin \xi_x(\Rset)$. Here $\xi_x = \xi(\cdot;x)$ is the orbit with $\xi_x(0)=x$.  Suppose that $\gamma:[0,T]\to  \Rset^2$ satisfies that $\gamma (0)=x$,  $\gamma (T)=y$  and
 $$
 \max_{t\in [0,T]}|H(\gamma(t))-H(x)|\leq \theta
 $$
 for some $\theta>0$.  Then there exists $t_0\in [0,T]$, such that 
 $$
 |V(\gamma(t_0))|\leq 3(K_0+1)\theta^{1\over 3}.
 $$
Here $K_0=\|V\|_{W^{1,\infty}(\Rset^2)}$.
 \end{lem}
 
 \begin{proof}  Assume that 
 \be\label{eq:nocritical}
 \gamma([0,T])\subset \Rset^2\backslash \Gamma.
 \ee
 Otherwise, the conclusion is trivial. For convenience, write $K_1=K_0+1$.   Without loss of generality, let $H(x)=H(y)=0$, then
 \be\label{eq:allequal}
 H(\xi(t;x))\equiv 0 \quad \text{for all $t\in \Rset$}
 \ee
 and
 \be\label{eq:thetabound}
 \max_{t\in [0,T]}|H(\gamma(t))|\leq \theta.
 \ee
For $t\in \Rset$, let $\eta_t(s)$ satisfy that 
 $$
 \begin{cases}
 \dot \eta_t(s)=DH(\eta_t(s))\\[3mm]
 \eta_t(0)=\xi_x(t).
 \end{cases}.
 $$
Clearly, if $\tilde t\in \Rset$    
satisfies $|DH(\xi_x(\tilde t))|=|V(\xi_x(\tilde t))|\geq 2K_1\theta^{{1\over 3}}$,
then for all $|s|\leq {\theta^{\wj{{1\over 3}}}\over K_1}$, we have 
\begin{equation*}
\begin{aligned}
|DH(\eta_{\tilde t}(s))| &\ge |DH(\eta_{\tilde t}(s)) - DH(\eta_{\tilde t}(0)) + DH(\xi_x(\tilde t))| \\
&\ge |V(\xi_x(\tilde t))| - K_0|\eta_{\tilde t}(s)-\eta_{\tilde t}(0)| \ge 2K_1\theta^{\wj{1\over 3}} - K_0^2 |s| \ge K_1\theta^{\frac13}.
\end{aligned}
\end{equation*}
By (\ref{eq:allequal}),  mean value theorem and the ODE satisfied by $\eta_{\tilde t}(\cdot)$ we deduce, for some $\lambda\in (0,1)$,
 \be\label{eq:positive}
 \left|H\left(\eta_{\tilde t}\left(\pm {\theta^{{1\over 3}} K_1^{-1}}\right)\right)\right| = \left|DH(\eta_{\tilde t}(\pm \lambda \theta^{1\over 3}K_1^{-1}))\right|^2 |\theta^{{1\over 3}} K_1^{-1}|\geq K_1 {\theta}>\theta. 
  \ee
 
 Since $H(y)=H(x)$ and  $H(\eta_t(s))\not=H(x)$ for $s\not=0$,  we have that $y\not= \eta_t(s)$ for all $t,s\in \Rset$. In particular, 
 $$
 y\notin \mathcal{C}=\left\{\eta_t(s)|\ t\in \Rset, \ s\in \left[- {\theta^{1\over 3}K_1^{-1}},  \ {\theta^{1\over 3}K_1^{-1}}\right]\right\}.
 $$
Owing to Lemma \ref{lem:threecases} and (\ref{eq:nocritical}),  we have either  $\overline {\mathcal{C}}=\mathcal{C}$ or $\overline {\mathcal{C}}\subset\mathcal{C}\cup \Gamma$. Thus,  there exist $t_0\in \Rset$ and $r_0\in (0,T)$  such that 
$$
\gamma(r_0)= \eta_{t_0}\left( \theta^{1\over 3}K_1^{-1}\right) \quad \text{or} \quad \gamma(r_0)= \eta_{t_0}\left( {-}\theta^{1\over 3}K_1^{-1}\right).
$$
Owing to (\ref{eq:thetabound}) and (\ref{eq:positive}), we must have
$$
|DH(\xi_x(t_0))| = |V(\eta_{t_0}(0))|\leq 2K_1\theta^{1\over 3}.
$$
This immediately leads to 
$$
|DH(\gamma(r_0))|\leq |V(\eta_{t_0}(0))| + K_0|\gamma(r_0)-\eta_{t_0}(0))| \le 3K_1\theta^{1\over 3}.
$$
The proof is hence complete.
\end{proof}
The following is an easy corollary.

\begin{cor}\label{cor:crosscell}  Let $S$ be a cell and  $\gamma:[0,T]\to \Rset^2$ be a control path satisfying $|\dot \gamma(t)+AV(\gamma(t))|\leq 1$ for a.e.\,$t\in \Rset$. Given $A\geq 3$, if $\gamma(0)$ and $\gamma(T)$ are both on $\partial S$ and are not on the same orbit, then at least one of the following holds:

\begin{enumerate}

\item[(1)]$T\geq {\log A\over A}$; or

\item [(2)] there exists $t_0\in [0,T]$ such that
$$
|V(\gamma(t_0))|\leq 3(K_0+1)\left({K_0\log A\over A}\right)^{1\over 3}.
$$
Here $K_0=||V||_{W^{1,\infty}(\Rset^2)}$.
\end{enumerate}

\end{cor}

\begin{proof}
    
Since $\gamma(0), \gamma(T)\in \partial S$, $H(\gamma(0))=H(\gamma(T))$.  If $T\leq {\log A\over A}$, then using the fact that $DH\cdot V = 0$ and via mean value theorem we have
\begin{equation*}
\begin{aligned}
H(\gamma(t_1))-H(\gamma (t_2)) &= DH(\gamma(t_*))\cdot \gamma'(t_*)(t_2-t_1) \\
&= DH(\gamma(t_*))\cdot [\gamma'(t_*)+AV(\gamma(t_*))](t_2-t_1),
\end{aligned}
\end{equation*}
for some $t_* \in [t_1,t_2]$, and, hence
\be\label{eq:smalldifference}
 |H(\gamma(t_1))-H(\gamma (t_2))| \leq {K_0\log A\over A }\quad \text{for $0\leq t_1\leq t_2\leq T$}.
 \ee
Then the desired result follows from Lemma \ref{lem:bound} by taking $\theta={K_0\log A\over A }$.
\end{proof}

\begin{figure}
\begin{center}
\includegraphics[scale=0.5]{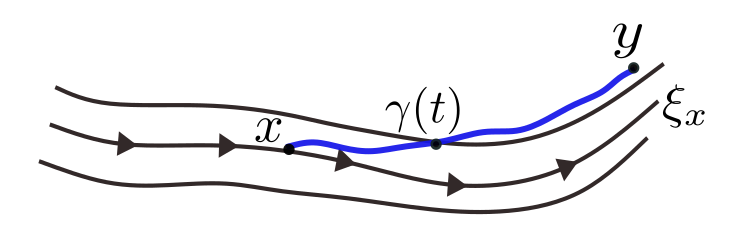}
\caption{\label{fig:boundaryC} $\gamma$ exists from the boundary of $\mathcal{C}$}
\end{center}
\end{figure}

 \section{Proof of Theorem \ref{theo:main}}
 
 If there is a non-contractible periodic orbit, then the conclusion follows immediately from (ii) of Theorem 1.2 in \cite{XY2014}. So we assume that there is no non-contractible periodic orbit.  Recall that $M$ is the bound on the diameter of swirls; see \eqref{eq:boundedswirl}. Hence the diameter of all cells is not greater than $M$. Throughout the proof,  $C$ represents a constant that depends only on $K_0= \|V\|_{W^{1,\infty}}$. 
 
 Without loss of generality, let $p=(-1,0)$. The arguments for $p=(1,0), (0,1), (0,-1)$ are similar. The upper bound for all unit direction $p$ follows from the convexity and homogeneity of $s_{\rm T}(p,A)$ with respect to $p$. 
 For $k\in \Nset$, write two lines
 $$
 L_1:\ \{x_1 = (k-1)M\}, \quad L_2:\ \{x_1 =(k+1)M+1\}. 
 $$
 
In view of the control formulation \eqref{eq:control} it suffices to show that if  $\eta$ is an admissible path connecting lines $L_1$ and $L_2$, that is, 
\begin{equation*}
\begin{aligned}
&\gamma \in W^{1,\infty}([0,T],\Rset^2), \; \gamma(0)\in L_1, \; \gamma(T)\in L_2,\\
&|\dot \gamma(t)+AV(\gamma(t))|\leq 1 \;\text{ for a.e. }\,  t\in [0,T],
\end{aligned}
\end{equation*}
then there exists a positive constant $C$ so that
 \be\label{eq:bound}
 T\geq {C\log A\over A}  \quad \text{for all $A\geq 3$}.
 \ee
  
 \medskip
 
 \begin{figure}
\begin{center}
\includegraphics[scale=0.4]{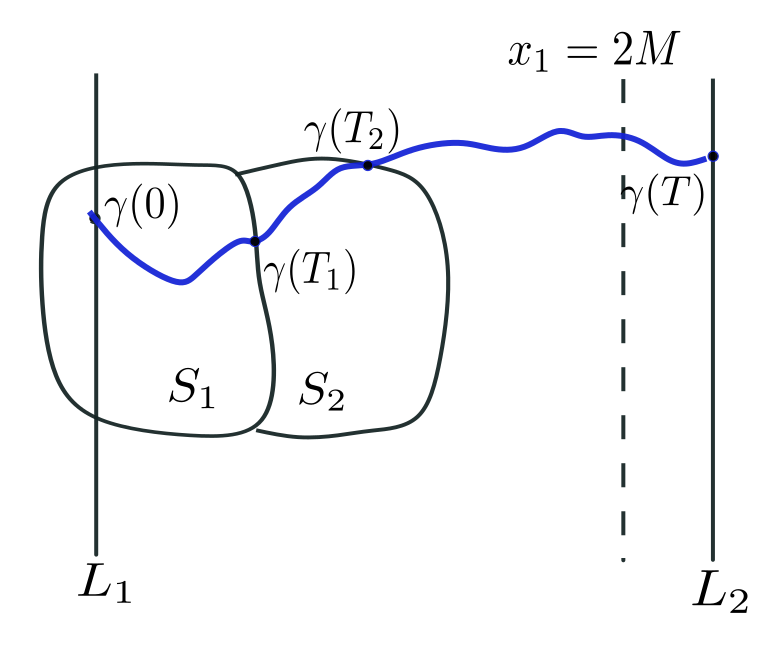}
\caption{\label{fig:control}Control path}
\end{center}
\end{figure}
\begin{proof}
    
Without loss of generality, let $k=1$. Fix $A\geq 3$.   If $T\geq {\log A\over A}$, we are done. So let us assume that $T\leq {\log A\over A}$. 

\medskip
 
 \emph{Claim}: \emph{There exists $t_0\in [0, T]$ such that
 $$
 |V(\gamma(t_0))|\leq 3(K_0+1)\left({K_0\log A\over A}\right)^{1\over 3}.
 $$
 }
 
Let us assume that $\gamma([0,T])\subset \Rset^2\backslash \Gamma$. Otherwise, the claim holds easily since $|V|$ vanishes on $\Gamma$. By Lemma \ref{lem:maximum},  let $S_1$ be the maximal cell that contains  $\gamma(0)$. Then $S_1\subset \{x_1\leq M\}$.  Denote the last exit time from $S_1$ by
 $$
 T_1=\max\{t\in [0,T]:\ \gamma(t)\in S_1\}.
 $$
 Then $\gamma(T_1)\in \{x_1\leq M\}\cap \partial S_1$ and $T_1 < T$; see Fig.\,\ref{fig:control}. Thanks to (3) of Lemma \ref{lem:property},  $\gamma(T_1)$ is on an orbit $\xi_1$ that is asymptotic to $\Gamma$. Owing to Corollary \ref{cor:adjacent},    let $S_2$ be the maximal cell that is adjacent to $S_1$ along $\xi_1$. Then $S_2\subset\{x_1\leq 2M\}$. Write $\bar x=\gamma(T_1)$. Owing to Lemma \ref{lem:onedirection}, we may assume that
 $$
U_{\bar x}^+(H(\bar x))\subset S_1^{\circ}\backslash S_2  \quad \mathrm{and}  \quad U_{\bar x}^{-}(H(\bar x))\subset S_2^{\circ}\backslash S_1. 
 $$
 Hence $\gamma$ must enter  $U_{\bar x}^{-}(H(\bar x))$ right after $T_1$. Write the last exit time from $S_2$ by
 $$
 T_2=\max\{t\in [T_1,T]:\ \gamma(t)\in S_2\}.
 $$
 Clearly, $T_2>T_1$,  $\gamma(T_2)\in  \{x_1\leq 2M\}\cap\partial S_2$ and $\gamma(T_2)\notin \xi_1(\Rset)$ due to the choice of $T_1$ and $\xi_1(\Rset)\subset S_1$; see Fig.\,\ref{fig:control}. Then our claim follows immediately from Corollary \ref{cor:crosscell}.

Now write $x_0=\gamma(t_0)$. We note that
$$
|V(x)| \le |V(x_0)| + K_0|x-x_0| \le C(\log A)^{1\over 3}A^{-{1\over 3}} + K_0|x-x_0|.  
 $$
Consider the function $r(t)=|\gamma(t)-x_0|$; we have
$$
 |r'(t)| \leq  |\gamma'(t)| \le 1 + A|V(\gamma(t))| \le C\left(Ar(t)+(\log A)^{1\over 3}A^{2\over 3}\right), \quad \text{for a.e. $t\geq t_0$}.
 $$
 Then for the path $\gamma$ to exit the ball $B_{1}(x_0)$, for $x_0 = \gamma(t_0)$, it takes an amount of time that is at least 
 $$
 \int_{0}^{1}{1\over C(\log A)^{1\over 3}A^{2\over 3}+CAs}\,ds \ge {C\log A\over A}.  
 $$
In view of $B_1(x_0)\subset \{x_1\leq 2M+1\}$, we see that \eqref{eq:bound} holds.
\end{proof}

\begin{rmk} Assume that $V$ does not have non-contractible periodic orbits.   If $\Gamma$ contains degenerate points,  $s_{\rm T}(p,A)$ might grow slower than $O(A/\log A)$.  If $V$ has no degenerate critical points, then using the nice structure established in \cite{A1991},  we can find a suitable control path to get the other direction $s_{\rm T}(p,A)\geq O(A/\log A)$. Hence $s_{\rm T}(p,A)=O(A/ \log A)$ in this situation.  Details are left to interested readers. 
\end{rmk}

\section{Other Examples} In this section, we look at several examples in the literature that possess Lagrangian chaos.

\subsection{Unsteady cellular flows} Let 
\begin{equation}
\label{eq:unsteady}
V(x,t)=(-H_{x_2}(x,t), H_{x_1}(x,t))
\end{equation}
for
$$
H(x,t)=U(t)\sin (2\pi (x_1+\alpha(t)))\sin (2\pi( x_2+\beta(t))).
$$
Here $U(t)$ is a periodic continuous function, $\alpha(t)$ and $\beta(t)$ are  periodic functions and H\"older continuous with exponent $r\in (0,1]$. Special cases like $\alpha(t)=B\sin (\omega t)$ and $\beta(t)=0$ have been considered in \cite{CTVV2003}.  The dependence of the turbulent flame speed $s_{\rm T}$ on the frequency $\omega$ are numerically investigated, which observed  interesting phenomena like frequency locking. Moreover, numerical results showed that the presence of Lagrangian chaos could either increase or decrease $s_{\rm T}$ compared to steady case ($B=0$) depending on values of the frequency $\omega$.  See also \cite{LXY2013} for similar discussions when $\alpha(t)$ is the Ornstein-Uhlenbeck process and $\beta(t)=0$.

\begin{theo} For the unsteady cellular flow \eqref{eq:unsteady}, change $V$ to $AV$ for $A>0$. Then the corresponding turbulent flame speed satisfies
\be\label{eq:time-upperbound}
s_{\rm T}(p,A)\leq {CA\over \log A}.
\ee
for any unit vector $p$ where $C$ is a constant depending only on $\|U(t)\|_{C^{0}(\Rset)}$, $\|\alpha(t)\|_{C^{0,r}(\Rset)}$ and $\|\beta(t)\|_{C^{0,r}(\Rset)}$.
\end{theo}

\begin{proof}The proof is a modification of the proof for the steady cellular flow. Let us prove the upper bound (\ref{eq:time-upperbound}) for $p=(-1,0)$. The arguments for $p=(1,0), (0,1), (0,-1)$ are similar. The upper bound (\ref{eq:time-upperbound}) for all unit direction $p$ follows from the convexity and homogeneity of $s_{\rm T}(p,A)$ with respect to $p$. 
Suppose $\gamma(t)=(x_1(t),x_2(t)) \;:\; [0,T]\to \Rset^2$ is a Lipschitz continuous control path satisfying
      \begin{equation*}
      |\dot \gamma+AV(\gamma(t),t)|\leq 1 \quad \text{for a.e. $t\in [0,T]$}
\end{equation*}
and for some $k\in \Nset$.
$$
x_1(0)+\alpha(0)=k,  \quad x_1(T)+\alpha(T)=k+1
$$
Since $\alpha(t)$ is uniformly bounded, owing to (\ref{eq:control}), it is enough to show that 
$$
T\geq {C\log A\over A} \quad \text{for all $k\in \Nset$}.
$$

If $T\geq {\log A\over A}$, we are done. So let us assume that $T\leq {\log A\over A}$. Then for $t\in [0,T]$,
$$
|\alpha(t)-\alpha(0)|\leq \|\alpha(t)\|_{C^{0,r}(\Rset)}\left({\log A\over A}\right)^{r}
$$
Choose $A_0\geq 2$ such that when $A\geq A_0$
$$
|\alpha(t)-\alpha(0)|\leq {1\over 2}.
$$
Let $s(t)=x_1(t)+\alpha(0)$. Then $s(0)=k$ and $s(T)\geq k+{1\over 2}$.  Note that for a.e. $t\in [0,T]$
\begin{equation*}
\begin{aligned}
|\dot s(t)| &=|\dot x_1(t)|\leq 1+ \left|AC\sin (2\pi(x_1(t)+\alpha(t)))\right|\\
&\leq 1+AC|s(t)|+C(\log A)^rA^{1-r}.
\end{aligned}
\end{equation*}
Accordingly, when $A\geq A_0$
$$
T\geq \int_{0}^{{1\over 2}}{1\over 1+ACs+(\log A)^rA^{1-r}}\,ds\geq {C\log A\over A}.
$$
We then conclude with the desired result as before.
\end{proof}

\subsection{ABC and Kolmogorov flows} In this section, we review two representative examples of three-dimensional periodic incompressible flows that have chaotic structures that have been well studied in the literature. See \cite{CG} for more details. 

\begin{itemize}
  \item[(1)] \emph{Arnold-Beltrami-Childress (ABC) flow}. The flow has the form
$$
V(x)=(a \sin x_3 + c \cos x_2,\ b \sin x_1 + a \cos x_3,\ c \sin x_2 + b \cos x_1),
$$
for $x=(x_1,x_2,x_3)$ and three fixed parameters $a$, $b$ and $c$.  ABC flows are steady solutions to the Euler equation and serve as natural models of turbulent flows. It is  probably the most studied example in this class. 

\item[(2)] \emph{Kolmogorov flow}. This flow is a variant of the ABC by removing the cosine term:
$$
V(x)=(\sin x_3, \  \sin x_1,\  \sin x_2).
$$
\end{itemize}

By establishing the existence of unbounded periodic orbits and the control formulation (\ref{eq:control}),  it has been shown in \cite{KLX,XYZ}, that the turbulent flame speeds for both 1-1-1 ABC flow ($a=b=c\not=0$) and the Kolmogorov flow grow linearly along all directions, i.e.,
$$
\lim_{A\to +\infty}{s_{\rm T}(p,A)\over A}=c_p  \quad \text{for all unit vector $p\in \Rset^3$.}
$$
Here $c_p$ is a positive constant depending on $p$.

\begin{rmk}\label{rmk:residue} In turbulent combustion, an important area of study is the behavior of the residual front speeds as the laminar flame speed approaches zero \cite{MK2019}, such as in lean or low-reactivity fuel mixtures. In this limit, the residual front becomes purely transport-dominated, driven by turbulence 
rather than the finite propagation speed of a laminar flame. For inviscid G-equation, 
the corresponding limit is:  
\be\label{eq:residue-front}
\lim_{s_l\to 0}s_{T}(p)>0  \quad \text{for all unit vector $p\in \Rset^3$.}
\ee
 Here $s_{T}(p)$ is the effective burning velocity associated with
$$
s_l|p+Dw|+V(x)\cdot(p+Dw)=s_T(p). 
$$
By setting $s_l=1$, changing $V$ to $AV$ and dividing $A$ on both sides, we see that the above limit (\ref{eq:residue-front})  is equivalent to  the linear growth of $s_T$: 
$$
\lim_{A\to \infty}{s_T(p,A)\over A} > 0 \quad \text{for all unit vector $p\in \Rset^3$.}
$$
The residual propagation phenomenon (\ref{eq:residue-front}) is also known as propagation anomaly, see \cite{MK2019} for related discussion about the anomaly 
and its absence if Yakhot's formula holds.  In reaction-diffusion models, due to the presence of molecular diffusion, residual front speed behaviors in chaotic and random flows are more diverse and observable thanks to the implicit square root relation with residual diffusivity in passive scalar models (\cite{BGK_98,TDMF_99,LXY2019,DEIJ2022} and references therein). The residual reaction-diffusion front speed asymptotics in the small molecular diffusion regime have been studied computationally in \cite{ResKPP_2015,ResKPP_2022,ResKPP_2025}.  

\end{rmk}

\section{Conclusions}
For two-dimensional mean zero spatially Lipschitz and 
periodic incompressible flows with bounded
swirls, we ruled out intermediate front speed growth laws between $O(A/\log A)$ and $O(A)$ in the regime of large flow amplitude $A$ of the inviscid $G$-equation. 
Our proof, based on the control 
representation, takes into account multi-scale cellular structures of the flow 
away from stagnation points and estimates travel times of the 
control trajectories across the cells. 
With the stagnation points allowed to be degenerate in this work,
the number of cells and their scales may be infinite within one period, which is a major difficulty we overcame. Lipschitz regularity 
corresponds to Batchelor limit in smooth turbulence flow (\cite{BGK_98,Son_99} and references therein).
If the flow has slightly lower regularity, e.g. half log-Lipschitz continuous, 
we recover Yakhot's $O(A/\sqrt{\log A})$ growth law with a modified cellular flow although its physical meaning is not very clear. Our proof also extends to unsteady smooth cellular flows with Lagrangian chaos and yields the $O(A/\log A)$ upper bound on the front speed for the first time. Here, mixing alone without roughness does not exceed the $O(A/\log A)$ law.
\medskip

Our results suggest that 
H\"older regularity and strongly mixing properties of rough turbulence flows may lead to an  
intermediate growth law between $O(A/\log A)$ and $O(A)$.  
The former tends to push front speed
above $O(A/\log A)$ while the latter
rules out $O(A)$ speed up from channel-like structures
(e.g. ballistic orbits in ABC and Kolmogorov flows \cite{KLX,XYZ}). A mathematical proof or counterexample of the $O(A/\sqrt{\log A})$ growth law
in the rough turbulence regime \cite{BGK_98} remains to be discovered.

\section{Acknowledgements}
This work was partially supported by NSF grants DMS-2000191 and 
DMS-2309520.

\bibliographystyle{plain}

\begin{thebibliography}{99}



\bibitem{A1991}{\sc V. Arnold},  {\em Topological and ergodic properties of closed 1-forms with incommensurable periods,},  Functional Analysis and Its Applications volume 25, pages 81--90 (1991).

\bibitem{BGK_98}
{\sc D. Bernard, K. Gawedzki, A. Kupiainen},
{\em Slow Modes in Passive Advection}, 
Journal of Statistical Physics, 90, 
pp. 519–569, 1998.


\bibitem{CNS} {\sc P. Cardaliaguet, J. Nolen, and P.E. Souganidis},
{\em Homogenization and enhancement for the G-equation}, Arch. Rational Mech and Analysis,
199(2), 2011, pp 527-561.

\bibitem{CTVV2003}{\sc M. Cencini, A. Torcini, D. Vergni and A. Vulpiani}, {\em Thin front propagation in steady and unsteady cellular flows}, Physics of Fluids 15, 679--688 (2003).

\bibitem{CG} {\sc S. Childress and A.D. Gilbert},
``Stretch, Twist, Fold: The Fast Dynamo'', Springer, 1995.



\bibitem{CL1983} {\sc M. G. Crandall,  P. L. Lions}, {\em Viscosity solutions of Hamilton-Jacobi equations},  Transactions of the American Mathematical Society,  Volume 277, Number 1, May 1983.




\bibitem{DEIJ2022} {\sc T. D. Drivas, T. M. Elgindi, G. Iyer, and I.-J. Jeong}, {\em Anomalous dissipation in passive scalar transport}, Archive for Rational Mechanics and Analysis, pages 1–30, 2022.


\bibitem{Evans_98}{\sc L.C. Evans},
{\em Partial Differential Equations}, Graduate Studies in Mathematics, AMS,
Providence, 1998.

\bibitem{TDMF_99}{\sc A. Fannjiang, T. Komorowski}, {\em Turbulent Diffusion in Markovian Flows}, Ann. Appl. Probab. 9:3 (1999), pp. 591-610.

\bibitem{HZ1994} {\sc P. Hill, D. Zhang},  {\em The effects of swirl and tumble on combustion in spark-ignition engines},  Progress in Energy and Combustion Science,  Volume 20, Issue 5, 1994, Pages 373--429.

\bibitem{KLX} {\sc C. Kao, Y.-Y. Liu, and J. Xin}, {\em A semi-Lagrangian computation of front speeds of G-equation
in ABC and Kolmogorov flows with estimation via ballistic orbits}, Multiscale Model. Simul.20 (2022), no. 1, pp. 107–117.

\bibitem{KAW1988} {\sc A. Kerstein, W. Ashurst, and F. Williams},  {\em Field equation for
interface propagation in an unsteady homogeneous flow},  Physical Review, 37(7):2728--2731, 1988.


\bibitem{LXY2013} {\sc Y. Liu, J. Xin and Y. Yu},  {\em Turbulent Flame Speeds of G-equation Models in Unsteady Cellular Flows}, Math Model. Natural Phenom., 8(3), pp. 198-205, 2013.

\bibitem{LXY2019} {\sc J. Lyu, J. Xin, and Y. Yu}, {\em Computing residual diffusivity by adaptive basis learning via
super-resolution deep neural networks}, in International Conference on Computer Science,
Applied Mathematics and Applications, Springer, New York, 2019, pp. 279--290.

\bibitem{ResKPP_2022}{\sc J. Lyu, Z. Wang, J. Xin, Z. Zhang},
{\em A convergent interacting particle method and computation of KPP front speeds in chaotic flows}, SIAM J. Numerical Analysis, 60(3), pp. 1136-1167, 2022.

\bibitem{M1951} {\sc  G. Markstein},  {\em Interaction of flow pulsations and flame propagation. Journal of the Aeronautical Sciences}, 18(6):428--429, 1951.



\bibitem{M1964} {\sc  G. Markstein},  {\em Nonsteady flame propagation (ed)}. AGARDograph,
Vol. 75, Elsevier, 1964.


\bibitem{MK2019}{\sc  J. R. Mayo, A. R. Kerstein},  {\em Log-Correlated Large-Deviation Statistics Governing Huygens Fronts in Turbulence}, Journal of Statistical Physicas (2019) 176: 456--477. 

\bibitem{MR2014} {\sc E. P. De Moura,  J.C. Robinson}, {\em Log-Lipschitz continuity of the vector field on the attractor of certain parabolic equations}, Dynamics of PDE, Vol.11, No.3, 211--228, 2014.


\bibitem{NXY}{\sc J.  Nolen, J. Xin,  and Y. Yu},
{\em  Bounds on Front Speeds for Inviscid and Viscous G-equations},
Methods and Applications of Analysis, Vol. 16, No. 4, pp 507-520, 2009.

\bibitem{O2001} A. Oberman, Ph.D thesis, University of Chicago, 2001.

\bibitem{OS1988}  {\sc S. Osher and J. Sethian}, {\em  Fronts propagating with curvature dependent speed: Algorithms based on Hamilton-Jacobi formulations},  Journal of Computational Physics, 79(1):12--49, 1988.


\bibitem{Pet00}{\sc N. Peters}, {\em Turbulent Combustion}, Cambridge University Press, 2000.



\bibitem{R}{\sc P. Ronney}, {\em Some Open Issues in Premixed Turbulent Combustion},
Modeling in Combustion Science (J. D. Buckmaster and T. Takeno, Eds.),
Lecture Notes In Physics, Vol. 449, Springer-Verlag, Berlin, 1995, pp. 3-22.



\bibitem{Son_99} {\sc D. T. Son},
{\em Turbulent decay of a passive scalar in the Batchelor limit:
Exact results from a quantum-mechanical approach}, 
Physical Review E, 59(4), R3811-R3814, 1999. 

\bibitem{W85}{\sc F. Williams},
{\em Turbulent Combustion}, The Mathematics of Combustion (J. Buckmaster, ed.),
SIAM, Philadelphia, pp 97-131, 1985.


\bibitem{XY_10}{\sc J. Xin and Y. Yu},
{\em Periodic Homogenization of Inviscid G-equation for Incompressible Flows},
Comm. Math Sciences, Vol. 8, No. 4, pp 1067-1078, 2010.


\bibitem{XY2013}{\sc J. Xin and Y. Yu},
{\em Sharp asymptotic growth laws of turbulent flame speeds in cellular flows by inviscid Hamilton-Jacobi models}, Annales de l'Institut Henri Poincar\'e, Analyse Nonlineaire, 30(6), pp. 1049--1068, 2013.

\bibitem{XY2014}{\sc J. Xin and Y. Yu},  {\em Asymptotic growth rates and strong bending of turbulent flame speeds of G-equation in steady two dimensional incompressible periodic flows}, SIAM J. Math Analysis, 46(4), pp. 2444--2467, 2014.

\bibitem{XYR2023}{\sc J. Xin, Y. Yu and P. Ronney},
{\em Lagrangian, Game Theoretic and PDE Methods for
Averaging G-equations in Turbulent Combustion:
Existence and Beyond},  Bulletin of the American Mathematical Society, 61(3), pp. 470-514, 2024. 


\bibitem{XYZ} {\sc J. Xin, Y. Yu, and A. Zlatos}, {\em Periodic orbits of the ABC flow with $A = B = C = 1$}, SIAM
J. Math. Anal. 48 (2016), no. 6, 4087–4093.

\bibitem{Yak}{\sc V. Yakhot},
{\em Propagation velocity of premixed turbulent flames}, Combust. Sci. Tech {\bf 60} (1988),
pp. 191-241.

\bibitem{ResKPP_2025}{\sc T. Zhang, Z. Wang, J. Xin, Z. Zhang}, 
{\em A convergent interacting particle method for computing KPP front speeds in random flows}, arXiv:2308.14479; to appear in SIAM/ASA Journal on Uncertainty Quantification, 2025.

\bibitem{ResKPP_2015}{\sc P. Zu, L. Chen, J. Xin}, {\em 
A Computational Study of Residual KPP Front Speeds in Time-Periodic Cellular Flows in the Small Diffusion Limit}, Physica D, Vol. 311-312, pp. 37-44, 2015. 


\end{thebibliography}

\end{document}